%% file: mader_large_pieces200212.tex
\documentclass[11pt]{article}

\input{minimalJ}

\usepackage{verbatim}
\usepackage{enumerate}
\usepackage[all]{xy}

\usepackage{subfig}

\title{\scshape Large highly connected subgraphs in graphs with linear average degree}

\author{Johannes Carmesin}

\newcommand{\sm}{\setminus}

\newcommand{\astr}{abstract separator-tree}

\begin{document}

\maketitle

\begin{abstract}
 In 1972 Mader proved that every graph with average degree at least $4k$ has a 
$(k+1)$-connected subgraph with more than $2k$ vertices. We improve this bound by showing that the 
constant $4$ can be replaced by $3+\frac{1}{3}$; this bound is sharp. 
\end{abstract}

\section{Introduction}

How can we find a highly connected substructure of a graph?
This is a fundamental question of Graph Theory that has many exciting answers depending on in which 
context we interpret `highly connected'?
Examples include Turan's theorem \cite{MR18405}, the tangle-tree theorem of Robertson and 
Seymour of their 
Graph 
Minor Theory \cite{GMX} or the following theorem of Mader.

In 1972, Mader proved that every graph\footnote{In this paper all graphs are finite. } with 
average degree at least $4k$ has a $(k+1)$-connected subgraph with more than $2k$ vertices 
\cite{Mader1972}, see also the textbook \cite{DiestelBookCurrent}[Chapter 1].
Seven years later, he conjectured that already average 
degree $3k$ should 
be enough to force a $(k+1)$-connected subgraph for large enough graphs\footnote{Actually he made 
a stronger conjecture proposing an explicit family of graphs to be extremal for large 
vertex numbers, see \autoref{concl_rem}.} \cite{Mader_survey}. 
Already in \cite{Mader1972} Mader proved a bound that is better than $4k$ for large graphs. Later 
this was in improved by Yuster \cite{Yuster_maderConj}. The current 
record is hold by 
 Bernshteyn and Kostochka \cite{Bernshteyn_Kostochka_maderConj} who proved the weakening with the 
bound 
$(3+\frac{1}{6})k$. 
Mader's theorem is applied in various contexts throughout Graph Theory, for example 
\cite{BOHME2009557} or 
\cite{THOMASSEN1983129};
\begin{comment}
 In the Kak paper they use a lemma that is slightly worse than Mader's original theorem. 
\end{comment}
more details can be found in Kriesell's survey  on connectivity 
questions in 
graphs \cite{Kriesell_minimalConny}. 

\vspace{.3cm}

\begin{comment}To complicated
 Here we improve Mader's theorem in a slightly different direction. It is reasonable to 
expect that $(k+1)$-connected subgraphs with more than $2k$ vertices are easier to find 
than smaller ones, but in return we are able to get a sharp result as follows. 
\end{comment}

Here we improve Mader's original theorem by replacing the constant $4$ by the optimal one. 

\begin{thm}\label{main-intro}
Every graph with average degree at least $(3+\frac{1}{3})k$ has a 
$(k+1)$-connected subgraph with more than $2k$ vertices.
\end{thm}

The constant $3+\frac{1}{3}$ in this theorem is best possible, see \autoref{extremal} for 
details and a discussion on the lower order terms. 

The edge-version of this question is somewhat different, see \cite{MR291004}.

Our proof can also be used to prove sharp versions of \autoref{main-intro} for $(k+1)$-connected 
subgraph with more than $m$ vertices, where $m\geq 2k$.
\begin{comment}
 \footnote{Why $2k$? One answer is that there 
seems to be a change in the structure of the extremal graph at $n=2k$ vertices, compare 
\autoref{small_graphs} and \autoref{sharp_con}.}
\end{comment}

\vspace{.3cm}

Mader's original argument makes use of a formula that relates the number of edges of a graph to the 
edges of two subgraphs obtained by cutting the graph at a separator. This formula also takes into 
account the edges in that separator. Mader's argument is suboptimal in that it does not analyse how 
many edges are in the separator. In our argument we make a more global analysis than Mader 
that allows us to estimate these edges, which in return gives the sharp bound.

The first step in our proof is to show that any graph without a $(k+1)$-connected 
subgraph of more than $2k$ vertices has a tree-like decomposition, which we call `separator-trees'. 
A key step is then to forget that we have a graph and just work in this more abstract 
framework. These abstract separator-trees have the advantage that they carry less information which 
enables us to show that in the `extremal case' they have a much simpler structure -- yet they 
attain the same bound.

\section{Separator-trees}\label{sec:sep-tr}

In this section we construct for every graph $G$ without a $(k+1)$-connected subgraph of more than 
$2k$ vertices a `separator-tree'. In this section we begin to analyse these tree-like structures in 
order to obtain an upper bound on the number of edges of the graph $G$; this is continued in the 
following sections.

Given a natural number $k$, a \emph{separator-tree} of a graph $G$ (of 
adhesion $k$) consists of a bipartite tree $T$ with the following properties, see 
\autoref{fig_sepr_tree}.
The nodes\footnote{In this paper we follow the convention that if a graph is a tree we refer to its 
vertices are `nodes'.} in the first bipartition class are called the \emph{parts} and labelled by 
induced 
subgraphs of $G$. All these nodes have degree one or two (except in the degenerated case, see 
below). One of these nodes has degree-one and is labelled 
by $G$. This part is referred to as the \emph{root}. All other leaves are called \emph{atoms}. 
In the degenerated case, the one where $T$ has only one node, that node is the root and atomic. 
We endow the tree $T$ by a tree-order in which the root is the largest element. 
The nodes in the second partition class are called the \emph{separator-nodes}. Each separator 
node has degree three. Each separator-node $t$ is labelled by a separator $S$ of at most $k$ 
vertices of the part $P$ directly above the node $t$. The two parts $A$ and $B$ directly below $t$ 
are sides 
of a separation $(A,B)$ of $P$ with separator $S$; in formulas: $A\cup 
B=P$, $A\cap B=S$ and no edge joins $A\sm 
S$ and $B\sm S$. 
We further require that all graphs assigned to atoms have at most $2k$ vertices. In a slight abuse 
of notation, we will not distinguish between the atoms and the graphs assigned to them and say 
things `the vertices of an atom'.

   \begin{figure} [htpb]   
\begin{center}
   	  \includegraphics[height=4cm]{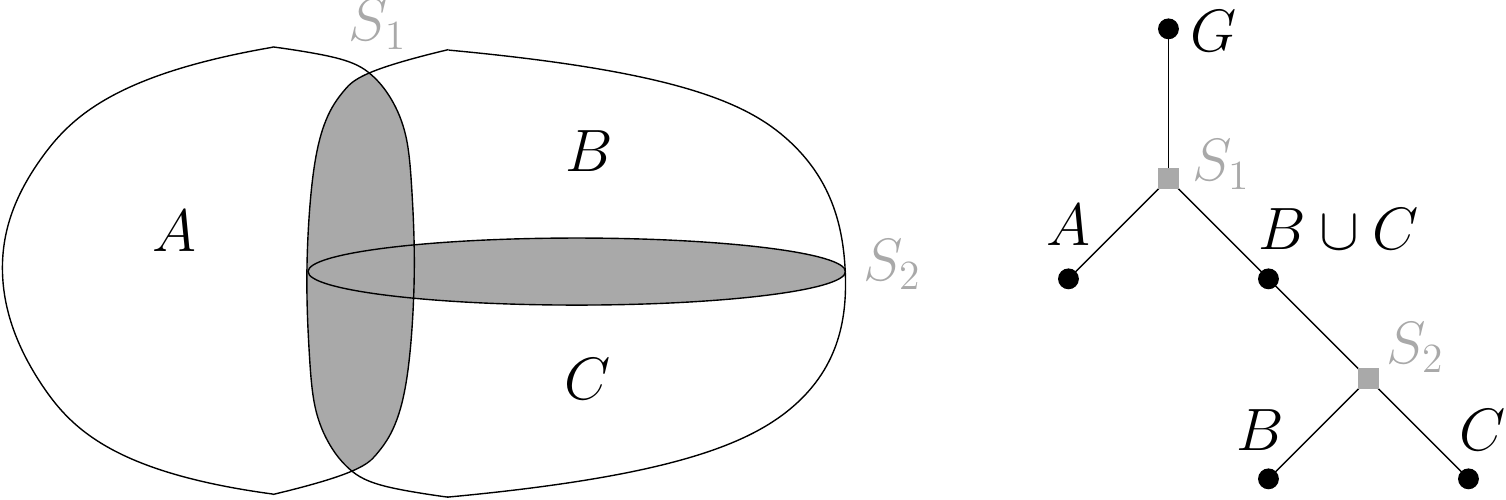}
   	  \caption{On the left we sketched a graph $G$ recursively separated 
by the separators $S_1$ and $S_2$, highlighted in grey. The corresponding 
separator-tree is depicted on the right. The three atoms are labelled by $A$, $B$ and $C$. This 
decomposition is not a tree-decomposition as the separator $S_2$ is not a separator of the graph 
$G$ but just of the subgraph $B\cup C$. }\label{fig_sepr_tree}
\end{center}
   \end{figure}

\begin{lem}\label{sep--tr_exist}
Any graph $G$ that has no $(k+1)$-connected subgraph with more than $2k$ vertices has a 
separator-tree of adhesion $k$. 
\end{lem}

\begin{proof}
We construct the separator-tree recursively. 
If $G$ has at most $2k$ vertices, we take the trivial separator-tree consisting only of one node.
Otherwise, by assumption $G$ is not $(k+1)$-connected. So it has 
a separation $(A,B)$ of order at most $k$ such that $A\sm B$ and $B\sm A$ are nonempty.
By recursion the subgraphs $G[A]$ and $G[B]$ have separator-trees with the desired properties.
We construct a separator-tree for $G$ as follows. We take the disjoint union of the two 
separator-trees for $G[A]$ and $G[B]$ and add a separator node adjacent to the two roots of these 
trees and the third neighbour gets the root of this separator-tree. The assignments of graphs to 
parts are the same as in the subtrees and we assign the graph $G$ to the root. The new separator 
node is labelled by the separator $A\cap B$. 
It is straightforward to check that this is a separator-tree.
\end{proof}

The \emph{branch} of a separator-tree $T$ of a part $P$ is the subtree of $T$ consisting of $P$ and 
all nodes in $T$ below $P$. In particular, it is also a separator-tree and all its atoms 
are also atoms of $T$. The \emph{branches at a separator} $S$ are the branches of the two parts 
directly below $S$ in the separator-tree. 

A \emph{valuation} of a separator-tree assigns to each separator a \emph{big branch}; the other 
branch at that separator is called the \emph{small branch}. The estimates below will depend on a 
choice of valuation. Throughout we assume that every separator-tree is endowed with a valuation, 
and say just things like `the big branch of a separator-tree'. 

Our aim will be to upper-bound the number of edges of any graph with the a separator-tree. We shall 
prove 
that bound recursively using the tree-structure of the separator-tree.
In the induction step we will rely on the following identity connecting the numbers of edges of a 
graph $G$ and the sides $A$ and $B$ of a separation $(A,B)$ with separator $S$, see 
\autoref{fig:separation}.

\begin{equation}\label{separation_edges}
 e(G)= e(A)+e(B)-e(S)
\end{equation}

   \begin{figure} [htpb]   
\begin{center}
   	  \includegraphics[height=2cm]{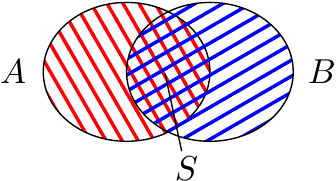}
   	  \caption{The separation $(A,B)$ with separator $S$. The edge number of the 
graph $G$ is that of the side $A$ plus that of the side $B$ minus the edges that are 
double counted; this are those in the separator $S$. }\label{fig:separation}
\end{center}
   \end{figure}
As mentioned above, it is key to our analysis to look carefully at the edges in the separators 
of the separator-tree. In fact, we find it slightly easier to analyse the edges that are not 
present: an \emph{anti-edge} of a graph  is an 
edge 
of its complement.

Given a separator $S$ of a separator-tree $T$ and a part $A$ directly below $S$ in 
$T$, and an anti-edge $x$ in $S$, we will recursively define when $x$ is `free' or `atomic' in $A$ 
at $S$.
In each recursion step, we have a part $A'$ containing the anti-edge $x$. We start with $A'=A$.
\begin{enumerate}
 \item If $A'$ is an atom, then $x$ is \emph{atomic} in $A$ within the atom $A'$;
 otherwise there is a unique separator node directly below $A'$ in $T$. We denote its separator by 
$S''$ and 
the two parts directly below by $A''$ and $B''$.
 \item If $x$ is an anti-edge between $A''\sm S''$ and $B''\sm S''$, we say that $x$ is \emph{free} 
in $A$.
\item If $x$ is an anti-edge of precisely one of the two parts $A''$ and $B''$, we iterate with 
that part and $x$.
\item Otherwise $x$ is an anti-edge of both $A''$ and $B''$. So $x$ is an anti-edge of $S''$. We 
iterate with the big branch of $S''$ and $x$. 
\end{enumerate}
This recursion clearly stops after finite time and it output is either that `$x$ is free in $A$ at 
$S$' or 
an atom $A'$ below $A$ such that  `$x$ is atomic in $A$ within the atom $A'$ at $S$', see 
\autoref{fig:atomic_free}.

   \begin{figure} [htpb]   
\begin{center}
   	  \includegraphics[height=6cm]{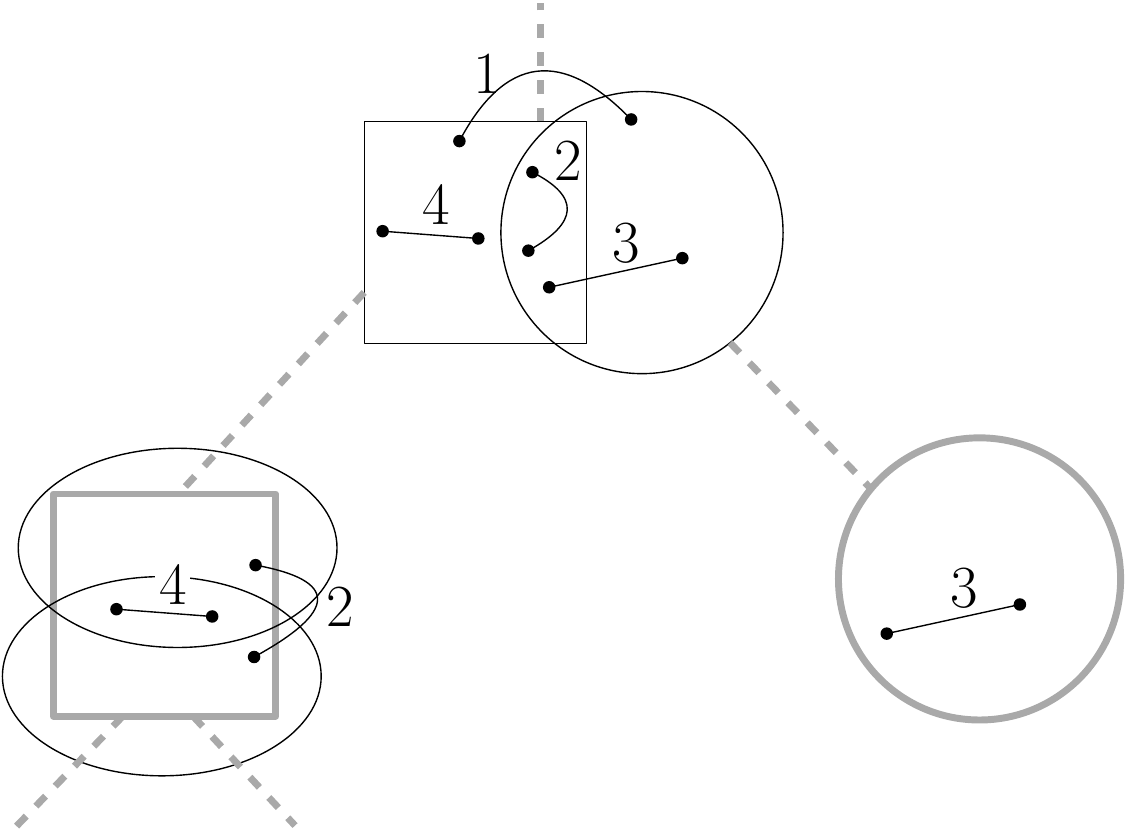}
   	  \caption{An image of a separator-tree. At the separator nodes we 
indicated the corresponding separations. At the top separator the left 
branch is the big one, the right branch only consists of an atom. The anti-edges labelled 1 or 2 
are free. The one labelled 3 is atomic. In order to find out what the anti-edge 4 is we have 
to look into the big branch of the bottom left separator. }\label{fig:atomic_free}
\end{center}
   \end{figure}
   
In the estimates in the proof of \autoref{main-intro} we will treat free anti-edges and atomic 
anti-edges differently. The existence of free anti-edges we have to accept but we can upper-bound 
their number using \autoref{septr_to_abstr} below. The atomic anti-edges can be used to improve 
our bounds. 
How much they improve the bound is measured in the `atomic defect' defined next. 

For that we define when an anti-edge $x$ is \emph{atomic within a branch $A$} (without reference to 
a particular separator). It will not be atomic in $A$ if there is no separator $S$ above $A$ 
containing $x$. If there is such a separator, let $S'$ be the one nearest to $A$, and let $A'$ be 
the branch of $S'$ containing $A$. We say that $x$ is \emph{atomic in $A$} if it is atomic in $A'$ 
at $S'$.  

The \emph{atomic defect} at a branch $A$ is the number of anti-edges that are atomic in the branch 
$A$; we denote this number by $\alpha(A)$. We stress that we count each anti-edge of $A$ at most 
once (although it may be in several atoms). 
By $\alpha(A,S)$ we denote the number of anti-edges that are atomic in the branch 
$A$ and have both end-vertices in $S$. 
\begin{eg}
 The atomic defect $\alpha(G)$ of the whole graph is zero. 
\end{eg}

\begin{lem}\label{alpha}
Let $P$ be a part and $S$ be the separator just below $P$ of the separator-tree $T$ and $A$ 
and 
$B$ be the two branches at $S$. Assume that $A$ is the small branch. Then: 
 \[
  \alpha(P)\leq \alpha(B)+\alpha(A)-\alpha(A,S)
 \]
\end{lem}

 \begin{proof}
Let $x$ be an anti-edge counted in the atomic defect $\alpha(P)$; that is, an atomic anti-edge of 
$P$. Then as $x$ is not free, it is an anti-edge of at least one of the subgraphs $A$ or $B$. 

If $x$ is an anti-edge of $B$, then it must be atomic in $B$ by the definition 
of `atomic'. If $x$ is, additionally to $B$, also an anti-edge of $A$, it is in the separator $S$. 
Thus it is counted in the atomic defect $\alpha(A)$ if and only if it is counted in $\alpha(A,S)$. 
To summarise, all anti-edges of $B$ that are counted on the left hand side, are also counted once 
in total on the right hand side.

It remains to consider anti-edges $x$ of $A$ without both endvertices in $S$. They are 
counted in the atomic defect $\alpha(P)$ if and only if they are counted in $\alpha(A)$. Hence 
every anti-edge counted on the left hand side of the inequality of \autoref{alpha} is counted once 
on the right hand side of this inequality. 
 \end{proof}
 
 \begin{rem}
  We remark that the 
inequality in \autoref{alpha} need not be an equality as the 
right hand side may additionally count some anti-edges of $S$ that are atomic in $B$ but not 
in $P$ as there is no separator above $P$ containing them.
 \end{rem}
 
Given a separator-tree $T$, its \emph{framework}, denoted by $\widehat T$, consists of the tree of 
$T$ 
and  it assigns a number 
$f(S)$ to every separator $S$ of $T$; we require that this number is always at least the number of 
free anti-edges at $S$ in the small branch of $S$. Then we use this to assign to each part-node 
$\widehat P$ of a part $P$ of $T$ an \emph{edge number} $e( \widehat P)$ as follows. If $P$ is an 
atom with $m$ vertices, we 
assign to it the number of edges of the complete graph on $m$ vertices. 
Then we use the following equation to define recursively the other edge numbers:

\begin{equation}\label{abstract_recursion}
 e(\widehat P)= e(\widehat A) + e(\widehat B)-\frac{s(s-1)}{2}+ f(S);
\end{equation}

here $s$ denotes the number of vertices of the separator $S$. 
Given a separator-tree $T$ with a framework $\widehat T$ and a part $P$ of $T$, by $\widehat P$ we 
denote the node $P$ considered as a node of (the tree of) the framework $\widehat T$. The following 
theorem compares the edge number of a part $P$ of a separator-tree $T$ with that of a part 
$\widehat P$ of a framework $\widehat T$ of $T$. 

\begin{thm}\label{baby-case}
 Let $G$ be a graph that has a separator-tree $T$ with a framework $\widehat T$. The number 
$e(P)$ of edges of a part $P$ is upper-bounded by
\[
 e(P)\leq e(\widehat P) - \alpha(P)
\]
\end{thm}

\begin{proof}%[Proof of \autoref{baby-case}.]
 We prove this by induction. If $P$ is an atom, then $\alpha(P)$ is the number of anti-edges of 
$P$ and hence we have equality in the inequality of \autoref{baby-case}.
Now let $P$ be a part that is not an atom and let $S$ be its top\footnote{The \emph{top 
separator} of a part $P$ that is not an atom is the separator just below $P$ in the tree-order. } 
separator.
Let $A$ and $B$ be the two branches of $S$. Recall that \autoref{separation_edges} is: 
$e(P)=e(A)+e(B)-e(S)$. Plugging in the induction hypothesis yields:
\[
 e(P)\leq e(\widehat A) - \alpha(A) + e(\widehat B) - \alpha(B) - e(S)
\]
We note that the number of anti-edges of the separator $S$ is at most $f(S)+ \alpha(A,S)$ by the 
definition of frameworks. 
So $e(S)\geq \frac{s\cdot (s-1)}{2}-f(s)-\alpha(A,S)$. 
Hence applying the recursive formula for frameworks, 
\autoref{abstract_recursion}, yields:
\[
 e(P)\leq e(\widehat P) - \alpha(A) - \alpha(B) +\alpha(A,S)
\]
Now we apply \autoref{alpha}: 
 \[
 e(P)\leq e(\widehat P) - \alpha(P)
\]
 This completes the induction step.  
\end{proof}

\section{Abstract separator-trees}

In this short section we define \astr s. 

Very roughly, abstract separator-trees are like separator-trees, where we forget (or `abstract') 
some of the information associated to the nodes. Abstract separator trees have the advantage that 
they are easier to analyse (for example they allow for more reduction operations\footnote{Compare 
the proof of \autoref{abstr_sep_tr} in 
\autoref{sec:gen_case}. There we make use of a reduction of an abstract separator-tree to a 
normal one. It is not clear how this reduction operation should be defined on the level of 
separator-trees. Moreover, \autoref{sec:saturn} below would be more technical 
otherwise. }) but 
yet -- as 
turns out -- they give the optimal bounds.

Given a natural number $k$, an \emph{ abstract separator-tree} (of 
adhesion $k$) consists of a bipartite tree $T$ with the following properties. 
The nodes in the first bipartition class are called the \emph{parts}. They all have 
degree one or two (except in the degenerated case, see 
below). One node of degree one is called the \emph{root}. All other leaves are called 
\emph{atoms}. 
In the degenerated case, the one where $T$ has only one node, that node is the root and atomic. 
We endow the tree $T$ by a tree-order in which the root is the largest element. 
The nodes in the second partition class are called the \emph{separator-nodes}. Each separator 
node has degree three. 
Furthermore each atom and separator-node is assigned an integer, which we refer to as its 
\emph{vertex-number}. Vertex numbers of atoms are between one and $2k$. Vertex numbers of 
separators are between zero and $k$. 

\begin{eg}
 Each separator-tree has an \emph{associated} abstract separator-tree. It has the same underlying 
tree but instead of assigning a graph to an atom it just assigns the number of vertices of that 
graph; similarly, it only assigns to a separator node the size of its separator. It does not 
assign anything to other nodes. 
\end{eg}

The \emph{vertex number} $n(P)$ of a part $P$ of an \astr\ $\widehat T$ is defined recursively on 
the 
tree order of $\widehat T$ as follows.
For atoms this number is already defined. So now let $P$ be a part that is not an atom. 
Let $S$ be the separator-node just below $P$ and let $A$ and $B$ be the two branches of $S$.
By recursion, we may assume that the vertex numbers of $A$ and $B$ are already defined. Let $n(S)$ 
be the vertex number of the separator $S$. We define.
\begin{equation}\label{vertex_recursion}
  n(P)=n(A)+n(B)-n(S)
\end{equation}

The \emph{vertex number} of the \astr\ $\widehat T$ is the vertex number of its root; we 
denote this number by $n(\widehat T)$.

\begin{eg}
For any graph $G$ with separator-tree $T$, the vertex number $n(\widehat T)$ of the associated 
\astr\ 
is equal to the vertex number of the graph $G$. 
\end{eg}

\begin{eg}
The restriction of an \astr\ to one of its parts (together with all 
nodes below) forms an \astr. 
\end{eg}

A key quantity of an \astr\ $\widehat T$ is its \emph{edge number} $e(\widehat T)$. The definition 
of this 
is slightly technical in the general case and will be revealed later. As mentioned earlier, it will 
be defined so that for 
any separator-tree $T$ its associated \astr\ is a framework; so the inequality of 
\autoref{baby-case} is 
satisfied.

An atom of an \astr\ (or a separator-tree) is \emph{normal} if its vertex number is at least 
$k+\frac{k}{3}$. An \astr\ is \emph{normal} if all its atoms are normal. 
A separator-node is 
\emph{saturated} if its vertex number is $k$. An \astr\ (or a separator-tree) is \emph{saturated} if 
all its atoms are 
saturated and every separator-node has a normal atom below. 
\begin{eg}
 All degenerated \astr s (the ones whose tree has only one node) are saturated. 
\end{eg}

In \autoref{sec:saturn} we prove the 
following. 
\begin{lem}\label{make_saturated}\label{saturn}
 For every real\footnote{See below for a definition.} \astr\ $\widehat T$, there is a saturated 
\astr\ $\widetilde T$ such that 
 \[
 e(\widehat T)\leq e(\widetilde T) + k \cdot (n(\widehat T)- n(\widetilde T))
\]
\end{lem}

\begin{rem}
\autoref{saturn} will be used to show that it will suffice to consider saturated \astr s. 
 \end{rem}

 We extend the definition of branches from separator-trees to \astr s in the obvious way. 
 A \emph{valuation} of an \astr\ picks at each separator a \emph{big} branch; the other branch is 
\emph{small}. 
 Throughout this paper we always pick a valuation such that the big branch has always at least as 
many normal atoms as the small one. For normal separator-trees we will not require any further 
properties. For \astr s that have atoms that are not normal we make a slightly more detailed 
choice, see below. 
\begin{comment}
 Alternative formulation
For a normal separator-tree $T$ throughout, we will pick the valuation that at a 
separator $S$ picks the big branch to be the one with more (normal) atoms; if both branches have 
the same number of atoms, we make an arbitrary choice. 
\end{comment}

\section{Normal atoms}\label{sec:normal_atoms}

In this section we prove \autoref{main-intro} for saturated normal separator-trees.
This special case will be used as a lemma in the general case.

Instead of proving \autoref{main-intro} directly, we use the following theorem as an intermediate 
step.

\begin{thm}\label{abstr_sep_tr}
Let $G$ be a graph with $n\geq 2k+1$ vertices and with a separator-tree $T$. Then $T$ has a 
framework 
$\widehat T$ whose edge number $e(\widehat T)$ is at most $\beta \cdot (n-k)+ 
\frac{\gamma}{L}-\epsilon$; 
where $L$ is the number of normal atoms of $\widehat T$ or one if that number is zero;
and $\beta=\frac{5}{3}k-\frac{1}{2}$ and 
$\gamma=\frac{k^2}{3}$ and $\epsilon=\frac{k}{2}$.
\end{thm}

\begin{proof}[Proof that \autoref{abstr_sep_tr} implies \autoref{main-intro}. ]
Let $k\geq 1$ be a natural number and $G$ be a graph that has 
average degree at least $(3+\frac{1}{3}) k$. Suppose for a contradiction that the graph $G$ has no 
$(k+1)$-connected subgraph with more than $2k$ 
vertices. Then by \autoref{sep--tr_exist} the graph $G$ has a separator-tree $T$. 
By the average degree condition, the vertex number $n$ of $G$ is at least $2k+1$.
Let $\widehat T$ be a framework for $T$ as in \autoref{abstr_sep_tr}. By \autoref{baby-case} the 
number of edges of $G$ is upper-bounded by the 
edge number $e(\widehat T)$. 

By \autoref{abstr_sep_tr} the 
edge number $e(\widehat T)$ is less than $\beta \cdot n$. So the graph $G$ has less than  $\beta 
\cdot 
n$ edges. 
It follows that the average degree of $G$ is less than $\frac{10}{3}k$. This is the desired 
contradiction. 
So the graph $G$ contains a $(k+1)$-connected subgraph with more than $2k$ 
vertices.  
\end{proof}

Let $T$ be a separator-tree with a separator $S$. Let $A$ be the small branch of $S$. We denote 
the number of atoms of $A$ by $L(A)$. We denote the number of vertices of $S$ by $s$. 

\begin{lem}\label{fS_normal}
 The number $F$ of free anti-edges in $A$ at $S$ is upper-bounded by  
 \[
 \frac{s^2}{2}- \frac{s^2}{2 L(A)};
\]
\end{lem}

 \begin{proof}
 Each vertex of $S$ is contained in at least one atom. Vertices that are contained in a common atom 
are not joined by a free anti-edge. We assign each vertex of $S$ to an atom containing it. The 
number $F$ of free anti-edges of $S$ in $A$, is upper-bounded by the anti-edges between vertices of 
$S$ assigned to different atoms. That is, it is upper-bounded by the number of edges of the 
complete multi-partite graph with the appropriate partition sizes. This number is maximised if 
 these partitions have the same size (up to rounding). It is well-known that the edge 
number 
of a complete balanced multi-partite graph with $s$ vertices and $L(A)$ classes is\footnote{One 
way to see this is to count the edges of its complement. It consists of $L(A)$ complete graphs, 
each on $\frac{s}{L(A)}$ vertices. }:
 \[
 \frac{s^2}{2}- \frac{s^2}{2 L(A)};
\]
(and it is a little smaller if there is rounding). This proves the lemma. 
 \end{proof}

 \autoref{fS_normal} motivates the following.
 We define \emph{the} framework $\widehat T$ \emph{associated} to a normal separator-tree $T$ by 
letting the value $f(S)$ at a separator $S$ be the term in \autoref{fS_normal}. (This is a 
refinement of the definition of `frameworks of separator-trees' given above). 
\begin{rem}\label{rem43}
 This definition could also be made if the separator-tree $T$ was not normal but it will not be 
good enough for us in this case and we will rely on a strengthening of \autoref{fS_normal} instead, 
see below.
\end{rem}

Similarly for normal \astr s $\widehat T$, we define a function $f$ assigning to each 
separator-node $S$ of $\widehat T$ the non-negative integer
\[
 f(S) =  \frac{s^2}{2}- \frac{s^2}{2 L(A)};
\]
here $s$ is the vertex number of $S$ and $L(A)$ is the number\footnote{If the number $L(A)$ is 
larger than 
$s$ this bound is not best-possible. Below this will make our estimate only a tiny bit worse as 
each time we apply it the values of $L(A)$ at least doubles. So this gives an error of at most 
one.} of normal atoms in the small branch 
$A$ of $S$. 
We refer to the number $f(S)$ as the \emph{number of free anti-edges} at $S$. 

The \emph{edge number} of a part $\widehat P$ of a normal \astr\ is recursively defined as follows.
The edge number of an atom with vertex number $m$ is equal to the number of edges of the complete 
graph on $m$ vertices; that is, it is equal to $\frac{m(m-1)}{2}$. 
As for frameworks, we define the edge number of a part $\widehat P$ recursively from the edge 
numbers of the two parts just below via \autoref{abstract_recursion}.
The \emph{edge number of an \astr} is the edge number of its root part. 

The definition of edge numbers of (parts of ) \astr s in the general case is the same except that 
the definition of the number $f(S)$ of free anti-edges is more technical (but specialises to the 
one given here in the normal case). We will reveal the definition of $f(S)$ in the general case 
later. 

\begin{eg}\label{eg44}
Given a separator-tree $T$, the edge number of its associated framework is equal to the edge number 
of the associated \astr\ of $T$. 
\end{eg}

The goal of this section is to prove the following strengthening of \autoref{abstr_sep_tr} for 
normal saturated
\astr s. In \autoref{sec:gen_case} we reduce the general case to this 
special 
case.

\begin{thm}\label{voll_normal}
Let $\widehat T$ be a nonempty normal saturated abstract separator-tree with vertex number $n$.
Then the edge number $e(\widehat T)$ is upper-bounded by $\beta \cdot (n-k)+ 
\frac{\gamma}{L}-\epsilon$ 
minus $\gamma$ times 
the branching error sum of $\widehat T$; 
\end{thm}

here $L$ denotes the number of normal atoms of $\widehat T$; and 
here the \emph{branching error sum} of $\widehat T$ is non-negative and defined to be the sum of 
the 
branching error terms of the separators of the $\widehat T$. The \emph{branching error term} 
at a separator $S$ of $\widehat T$ is: 
\begin{equation}\label{eq6}
  x(S) = \frac{1}{2} \frac{1}{\ell(S,-)}+ \frac{1}{\ell(S,-)+\ell(S,+)}-\frac{1}{\ell(S,+)};
\end{equation}
where $\ell(S,-)$ is the number of normal atoms of the small branch of $S$ and similarly 
$\ell(S,+)$ is the number of normal atoms of the big branch of $S$. In particular, $x(S)$ is zero 
if these two numbers agree. 

Throughout the paper, we denote the number of normal atoms below a part $P$ by $L_P$.

Before we start with the proof, we show the following simple lemma.
Let $P$ be a part of  a separator-tree that is not an atom, $S$ be its top separator 
and $A$ and $B$ be the two branches of $S$. 

\begin{lem}\label{calc1}
Assume that the branch $A$ is small. Then:
\[
 \frac{1}{L_A}+ \frac{1}{L_B}\leq \frac{1}{L_P} + \frac{3}{2}\cdot \frac{1}{L_A}  
\]
\end{lem}

\begin{proof}
Recall that $L_P=L_A+L_B$. We abbreviate $a=L_A$ and $b=L_B$. 
As $a,b\geq1$ and $a\leq b$ we estimate:
\[
 \frac{1}{b}- \frac{1}{a+b} = \frac{a}{b\cdot (a+b)}\leq \frac{1}{a+b}\leq 
\frac{1}{2}\cdot \frac{1}{a}
 \]
 
The inequality between the most left and the most right term can be rearranged to the inequality 
of the lemma. 
\end{proof}

\begin{proof}[Proof of \autoref{voll_normal}.]
  Let $\widehat T$ be a nonempty normal saturated \astr\ and $\widehat P$ be one of branches part. 
We prove 
\autoref{voll_normal} by induction for all branches $\widehat P$ of $\widehat T$ using the 
tree structure of $\widehat T$. 
 
 First we prove the weakening of \autoref{voll_normal} where we ignore the error term; that is, we 
show that 
 $e(\widehat P)\leq \beta \cdot (n-k)+ \frac{\gamma}{L}-\epsilon$ for all parts $\widehat P$ of 
$\widehat 
T$. (In the last step of the induction we prove that inequality for $\widehat P=\widehat T$).
 
 The base case is that the branch $\widehat P$ consists of only one atom. 
 At this point we check the base case but reveal it later; this way it will be easier to 
understand for the reader why we chose the constants $\beta$, $\gamma$ and $\epsilon$ as we did.

Now we start with the induction-step. So let $\widehat P$ be a branch whose abstract separator-tree 
has 
at 
least one separator node. Let $S$ be the top separator of $\widehat P$. Let ${\widehat 
A}$ and ${\widehat B}$ be the two branches of $S$. 
We apply the induction hypothesis to the \astr\ of the branches ${\widehat A}$ and ${\widehat 
B}$. Plugging that into \autoref{abstract_recursion} yields:
\begin{equation}\label{long}
  e(\widehat P)\leq \beta \cdot (n_{\widehat P}-k)+ \gamma \cdot 
\frac{1}{L_{\widehat A}}+ 
\gamma 
\cdot 
\frac{1}{L_{\widehat B}}-2\epsilon- \frac{k(k-1)}{2}+ f(S);  
\end{equation}
here recall that the separator $S$ has vertex number $k$ as $\widehat T$ is saturated.

We shall show that the term on the right hand side is bounded from above by $\beta \cdot 
(n_{\widehat 
P}-k)-\epsilon+ \frac{\gamma}{L_{\widehat P}}$. Since the involved expressions are quite long, we  
compare 
them in parts. We start by comparing the `$\gamma$'-terms; that is, the terms with the coefficient 
$\gamma$.

By symmetry we may assume that ${\widehat A}$ is the small branch at $S$ and ${\widehat B}$ is the 
big 
branch. 
By \autoref{calc1}, we have:

\begin{equation}\label{s1}
\gamma \cdot \frac{1}{L_{\widehat A}}+ \gamma \cdot 
\frac{1}{L_{\widehat B}}\leq \gamma \cdot\frac{1}{L_{\widehat P}} + \frac{3}{2}\cdot \gamma \cdot 
\frac{1}{L_{\widehat A}}
\end{equation}

In short, in comparing the $\gamma$-terms we have the additional summand $\frac{3}{2}\cdot \gamma 
\cdot \frac{1}{L_{\widehat A}}$. By definition of $f(S)$, the sum of the last two terms of 
\autoref{long} is $- \frac{k^2}{2 \cdot 
L_{\widehat A}}+\frac{k}{2}$. The difference of the $\epsilon$-terms is $-\epsilon$.
We will show that the sum of these three terms evaluates to zero; that is;
\[
 0= \frac{3}{2}\cdot \gamma 
\cdot \frac{1}{L_{\widehat A}}- \frac{k^2}{2 \cdot 
L_{\widehat A}}+\frac{k}{2}-\epsilon
\]
Our choice that $\epsilon=\frac{k}{2}$ ensures that the last two terms cancel. The other two terms 
cancel as $\gamma= \frac{k^2}{3}$. 
So $e(\widehat P)\leq \beta \cdot (n_{\widehat P}-k)-\epsilon+ \frac{\gamma}{L_{\widehat P}}$.
This completes the induction step.

\vspace{.3cm}

Having finished the induction-step, we now reveal the induction-start. 
We abbreviate $n_{\widehat P}$ simply by $n$ for the rest of this proof. 
Here
\[
 e({\widehat P})=\frac{n(n-1)}{2}
\]
So it remains to show that
\[
 \frac{n(n-1)}{2}\leq \beta\cdot  (n-k)+ \frac{k^2}{3}-\frac{k}{2}
\]
for $n$ between $2k$ and $k+\frac{k}{3}$.

We remark that the choice of $\beta=\frac{5}{3}\cdot k- \frac{1}{2}$ implies that 
for 
$n=2k$ we have equality in the above inequality (to see that we first ignore the linear term; 
that is a term of $-\frac{n}{2}$ on the left and a term of $- \frac{1}{2} (n-k)$ and 
-$\frac{k}{2}$ on the right. Then we note that these linear error terms cancel for $n=2k$). We 
consider $k$ as fixed and let $n$ vary. If the absolute value of $n$ is huge, the above inequality 
is false. As it has degree two in $n$, it is true in between the two values for $n$ where we have 
equality. Aside from $n=2k$, the other value is $n=k+\frac{k}{3}$ (if we ignore the linear error 
terms. If not, we verify that they cancel.). This completes 
the induction start. 
 
 \vspace{.3cm}
 
 This completes the proof that  $e(\widehat P)\leq \beta \cdot (n-k)+ \frac{\gamma}{L}$. Now we 
explain 
how this proof can be modified to obtain the strengthening claimed in \autoref{voll_normal}. 
In some of the induction-steps we wasted a little bit when applying \autoref{calc1}. The terms we 
wasted in the induction-step at the separator $S$ is equal to the error term $x(S)$ at $S$. Taking 
these terms into account proves the stronger version of \autoref{voll_normal}. 
\end{proof}

 \vspace{.3cm}

For later reference, we estimate the following.
\begin{lem}\label{calci}
 If $\ell=\ell(S,-)$ is strictly smaller than $\ell(S,+)$, then 
\[
 x(S)\geq \frac{1}{4\ell^2}
\]

\end{lem}
\begin{proof}
We obtain a lower bound for $x(S)$ by replacing  $\ell(S,+)$ by $\ell+1$ in \autoref{eq6}. 
Elementary computations yield:
 \[
  x(S)\geq \frac{1}{2} \left[    \frac{(\ell+1)\cdot (\ell+0.5)-\ell^2}{\ell\cdot (\ell+1)\cdot 
(\ell+0.5)}                \right]
  \]

We simplify:
 \[
  x(S)\geq \frac{1}{4}\cdot\frac{3\ell+1}{\ell\cdot (\ell+1)\cdot (\ell+0.5)}
 \]
 We further estimate:
 \[
  x(S)\geq \frac{1}{4}\cdot\frac{3}{(\ell+1)\cdot (\ell+0.5)}\geq 
\frac{1}{4}\cdot\frac{3}{3\ell^2}= \frac{1}{4\ell^2};
 \]
here we use that $\ell\geq 1$ in the second inequality as this is true for any separator of an 
\astr.\footnote{This rough estimate could be improved to give a better constant. We will not use 
this later.}
 \end{proof}

 \vspace{.3cm}
 
 We end this section by giving several definitions related to \astr s that we have defined above 
only in 
the normal case. 

Recall that a branch of a separator\footnote{For simplicity, we will occasionally abbreviate 
`separator-node' by `separator'. Mostly we use this in the context of \astr s, where is no 
danger of confusion.} $S$ is big if it has more normal atoms than the other branch of 
$S$. A separator is \emph{normally balanced} if both its branches contain the same number of normal 
atoms. 

A branch $A$ of a separator $S$ is \emph{almost small} for $S$ if $A$ does not 
contain more normal atoms than the other branch of $S$. 
\begin{eg}
 Both branches of a normally balanced 
separator are almost small.
\end{eg}
Now we define how we choose the big branch of normally balanced 
separators. This is defined by recursion on the tree-structure of its separator-tree. 

An atom that is not normal is \emph{tiny}. 
A \emph{tiny vertex} of a separator-tree is a vertex in a tiny atom that is not in the separator 
just 
above its atom. The \emph{number of tiny vertices in a tiny atom} is the vertex number of that atom 
minus the vertex number of the separator just above. 
\begin{rem}\label{tiny_vertices}
 Although tiny vertices are just defined for separator-trees, the number of tiny vertices is also 
defined for \astr s. 
\end{rem}

A tiny atom $A$ \emph{desires} a separator $S$ if the branch of $S$ containing $A$ is almost small 
for $S$ and there is no separator below $S$ 
that has a small 
branch containing $A$ other than the separator just above $A$, and $S$ is not the separator just 
above $A$. 
The \emph{number of tiny vertices relevant} at the separator $S$ in a branch $A$ of $S$
is the sum of numbers of tiny vertices of tiny atoms that desire the separator $S$ and are 
in the branch $A$. 

We choose 
the \emph{big branch} of a normally balanced separator in an \astr\ (or a separator-tree)  such 
that its number of relevant tiny vertices is at least as big as that of the other branch; if both 
branches have the same number of 
relevant tiny vertices, we make the choice arbitrarily. 

We shall assume throughout that every \astr\ (or separator-tree) has a valuation with 
this property and we will not 
always mention valuations explicitly but assume that they are implicitly given by the context and 
just say things like `the big branch of a separator $S$'.

\vspace{.3cm}

Given a separator-tree, the anti-edges incident with a vertex $v$ of an atom $A$ contribute 
at at most one separator to the number of 
free-anti-edges. This separator is the smallest separator above $A$ such that 
$A$ is in its small branch and the vertex $v$ is in the separator; if there is no such separator, 
then the anti-edges incident with 
$v$ never contribute as free anti-edges. This separator could be the one desired by $A$ or be 
above it.
Next we define the `technical data' of an \astr, which is a function that mimics this for 
\astr s.

Given an \astr, a \emph{technical data} is a function $\iota$ assigning a number $m(A,S)$ to each 
pair $(A,S)$ consisting of a tiny atom $A$ in the small branch of $S$, where $S$ is not the 
separator just above $A$. This function $\iota$ satisfies the following:
\begin{itemize}
 \item for a fixed tiny atom $A$ the sum ranging over all numbers $m(A,S)$ is upper-bounded by the 
number of tiny vertices of $A$;
 \item for a fixed separator $S$, the sum ranging over all numbers $m(A,S)$ is upper-bounded by the 
vertex number of $S$. 
\end{itemize}

We abbreviate the sum over all $m(A,S)$ with $S$ fixed by $m(S)$. 

\begin{eg}
 Let $\widehat T$ be the associated \astr\ of a separator-tree $T$. The \emph{associated} technical 
data is defined as follows.
For each tiny vertex $v$, let $S_v$ be the smallest separator of $T$ such that the atom of $v$ is 
in the small branch of $S_v$ and $v\in S_v$. 
Note that $S_v$ is not the separator just above the atom of $v$. Note further that $S_v$ may be 
undefined. 
We let $m(A,S)$ be the number of tiny vertices $v$ of the atom $A$ such that $S=S_v$. This 
completes the definition of the associated technical data. 
\end{eg}

\begin{rem}
 We will assume that technical datas are always implicitly given by the context and not mention 
them explicitly. They will only appear explicitly at the end of the proof of 
\autoref{molecule_removal}. Intuitively, one of the things this proof shows is that 
one may always assume that the technical data assigns the number of tiny vertices at an atom $A$ to 
the separator desired by $A$ (if existent). 
\end{rem}

The \emph{number $f(S)$ of free anti-edges} at a separator $S$ of an \astr\ is defined as 
follows.
\[
   f(S)=\frac{s^2}{2}- \frac{(s-m(S))^2}{2 L_A} - \sum \frac{(m(X,S))(m(X,S)-1)}{2}- \frac{m(S)}{2};
\] 
here $L_A$ denotes the number of normal atoms in the small branch $A$ of $S$; and here the later 
sum 
reaches over all tiny atoms $X$ in $A$ and $s$ denotes the vertex number of $S$.

 \begin{lem}\label{septr_to_abstr}
  Let $T$ be a separator-tree and $\widehat T$ its associated \astr. Let $S$ be a separator.
  Then the number of free anti-edges at $S$ in $T$ is upper-bounded by the number $f(S)$ of 
$\widehat 
T$. 
 \end{lem}

 \begin{proof}
 First we prove the special case where no tiny vertex is assigned to $S$. 
 This case is just \autoref{fS_normal}. 

Each vertex of the separator $S$ is in several atoms. Similarly, as in the definition of `free 
anti-edges', we now assign to each vertex $x$ of $S$ a vertex $x'$ of an atom such an anti-edge 
$vw$ is free if and only if the vertices $v'$ and $w'$ are not in the same atom. 
Let $x$ be a vertex of the separator $S$. Let $P$ be the following path of the tree $T$ starting 
from the small branch at $S$ going downwards in the tree-order. For that we specify which direction 
$P$ takes at all degreee 3 vertices of $T$; that is, at the separator nodes. If there is a unique 
branch containing the vertex $x$, the path $P$ goes in that direction. If the vertex $x$ is in both 
branches, the path $P$ goes in the direction of the big branch. The last vertex of the path $P$ is 
an atom.  
The vertex $x'$ is the copy of the vertex $x$ in the atom that is the last vertex of the path $P$. 
It follows immediately from the definition of `free' that an anti-edge 
$vw$ is free if and only if the vertices $v'$ and $w'$ are not in the same atom. 
If follows from the definition of associated technical data that for any tiny atom $A$ precisely 
$m(S,A)$ vertices $x$ of $S$ have the copy $x'$ chosen from the atom $A$.  

Thus we have $s-m(S)$ vertices of $S$ in normal atoms and in addition $m(S)$ tiny 
vertices. It is straightforward to check in this case that the number of free anti-edges is 
upper-bounded by $f(S)$.   Indeed, the number of free anti-edges is upper bounded by
\[
 \frac{s(s-1)}{2}- L_A\cdot \frac{\frac{s-m(S)}{L_A}\left( \frac{s-m(S)}{L_A}-1   \right)}{2} 
- \sum \frac{(m(X,S))(m(X,S)-1)}{2}
\]
It is a routine computation that this evaluates to the term $f(S)$. 
 \end{proof}
 
\section{Saturating separators}\label{sec:saturn}

The purpose of this (short but slightly technical) section is to prove \autoref{saturn}, which is 
used in the proof of 
\autoref{main-intro}. Basically this lemma says that in the proof it suffices to consider the case 
where the separators do not have less than $k$ vertices. 

Given an \astr\ $\widehat T$, a \emph{saturation} is obtained by doing successively one of the 
following operations.

\begin{enumerate}
\item Assume there is an atom $A$ whose vertex number is not more than that of the 
separator-node $S$ just above. 
The technical data cannot assign tiny vertices to the separator $S$ (that is\footnote{We use 
`$m(\bullet, S)$' as an abbreviation for `$m(A, S)$ for some $A$'.
}, all 
$m(\bullet, S)$ are zero); indeed, if the other branch has tiny vertices, then that branch is big 
and cannot assign vertices to $S$.
In this case, we delete the atom $A$.  By that the separator-node $S$ get 
degree two. Then we suppress\footnote{that is, we replace that vertex by an edge.} the node 
$S$. 
 \item Assume there is a separator-node $S$ that has vertex number less than $k$ and 
smaller than that of its small branch. Then we increase the vertex number of $S$ by one. 
If the small branch consists of a single tiny atom $A$, it may be necessary to also modify the 
technical data. If so, one of the values $m(A,\bullet)$ is positive and we 
decrease it by one. 
 \item Assume there is a part $P$ with at most $2k$ vertices that is not an atom and that has no 
normal 
atom below. Then delete all nodes below it.
If the part $P$ has less than $k+\frac{k}{3}$ vertices, then we define the numbers $m(P,S)$ of the 
technical data to 
be the sums of the numbers $m(A,S)$, where $A$ is below $P$. (And we delete the terms $m(A,S)$ 
from the technical data).
If the part $P$ has more than $k+\frac{k}{3}$ vertices, then we just delete the terms $m(A,S)$ for 
the atoms $A$ below $P$. 

\end{enumerate}

Here we only do the second if we cannot do the first one; and we only do the third if we cannot do 
the first or the second, and we choose the part $P$ in operation 3 minimal in the tree-order. 
This process has to terminate eventually as each step decreases the vertex number $n(\widehat T)$ 
or decreases the number of vertices of the tree $\widehat T$ and no operation increases any of 
these numbers. 
Hence every \astr\ $\widehat T$ has a saturation $\widetilde T$. 

\begin{lem}\label{is_astr}
 The saturation $\widetilde T$ is a saturated \astr.
\end{lem}

\begin{proof}
First we show that the data defined in each performance of the third operation is indeed a 
technical data. For that we have 
to show that the number of tiny vertices of $P$ is not smaller than the sum of tiny vertices of all 
tiny atoms in $P$. If all separators in $P$ have precisely $k$ vertices, this is clear. 
Suppose for a contradiction that the part $P$ has a separator $S$ whose vertex number is less than 
$k$. Let $A$ be the small branch of $S$. If $A$ had higher vertex number than the separator $S$, we 
could perform the second operation instead. So the branch $A$ 
must have vertex number no bigger than that of the separator $S$. If the branch $A$ was an atom, we 
could perform the first operation instead. Hence it is not an atom. But then the branch $A$ 
contradicts the minimality of $P$. Hence all separators of $P$ have size $k$ and the data defined 
in the third operation is indeed a technical data.  

Thus saturations of \astr s are \astr s (whose technical datas are modified as indicated above).

\vspace{.3cm}

Towards showing that they are saturated,
next we prove that every separator-node has vertex number $k$. Suppose not for a contradiction. 
Amongst the separator-nodes whose vertex number is smaller than $k$, we pick $S$ so that it is 
smallest in the tree-order. Let $s$ be its vertex number. Since we cannot apply operation 2, its 
small branch has vertex number at most $s$. As we cannot apply operation 3, this branch must 
be an 
atom. Hence we can apply operation 1 to this part. This is a contraction to the assumption that 
$\widetilde T$ is a saturation. 
Hence every separator-node has vertex number precisely $k$.

It remains to show that below every separator node $S$ there is a normal 
atom. By induction, it suffices 
to show this for separator-nodes that do not have any other separator nodes below them. Let $S$ be 
such a 
separator node. Let $P$ be the part just above $S$. As $\widetilde T$ is a saturation, we cannot 
apply operation 3. Hence we may assume that the vertex number $n(P)$ of $P$ is at least $2k+1$. Let 
$A$ and $B$ be 
the 
two atoms below the separator $S$. By the recursion formula for the vertex number, 
\autoref{vertex_recursion}, and since $S$ has vertex number $k$ one of the vertex numbers of $A$ 
and 
$B$ has to be at least $\frac{3k+1}{2}$. Hence one of the atoms $A$ or $B$ is normal. 
\end{proof}

An \astr\ is \emph{real} if the vertex number of each of its separators is upper-bounded by the 
minimum of the vertex numbers of its two branches. 
\begin{eg}
 All \astr s associated to separator-trees are real. 
\end{eg}

\begin{lem}\label{saturn1}
 Let $\widetilde T$ be a saturation of a real \astr\ $\widehat T$. Then 
  \[
 e(\widehat T)\leq e(\widetilde T) + k \cdot (n(\widehat T)- n(\widetilde T))
\]
\end{lem}

\begin{proof}
First we show that being real is preserved by the performed operations. The only operation that 
could destroy that property is the second operation. It only would destroy it if there were a 
separator $S$ that had the same vertex number as one of its branches $A$. As the \astr\ is real, 
choosing the separator $S$ minimal in the tree-order ensures that the branch $A$ must be an atom. 
Hence we can perform the first operation. By the definition of saturation, we do the first operation 
instead of the second. 
Hence being real is preserved by the performed operations. 

It is easy to check that each time we apply the first operation, the edge number can only go up. 

Next we show that each time we apply the third operation, the edge number can only go up.
If the part $P$ has vertex number less than $k+\frac{k}{3}$, this is clearly true. Indeed, the 
increase in the edge number of $P$ compensates for the loss of free anti-edges. 
So we may assume that the part $P$ has at least that many vertices. 
Moreover, 
if we would do the construction from operation 3 for $P$ as in the tiny case, the same argument 
as above shows that the edge number can only go up. 
Using the formula for the number of free anti-edges (this is defined just before 
\autoref{septr_to_abstr}), it is a simple computation that increasing 
$L_A$ by one but deleting the term $m(P,\bullet)$ does not decrease the edge number. Indeed, 
before we count the edges of some multipartite graph with vertex number 
$n(\bullet)-m(\bullet)+m(P,\bullet)$ and after that we count the edges of a balanced multipartite 
graph with the same vertex number and partition classes. As the balanced multipartite graph 
maximises this number, the edge number can only go up.
Hence in the case that the part $P$ is not tiny, the modifications can be done in two steps and in 
each step the edge number does 
not go down.

So it suffices to show that each time we apply the second operation, the above 
inequality of \autoref{saturn1} remains true. 
When applying operation 2, the vertex number goes down by one. The only way the edge number can 
change is that the terms $f(S)$ for the free anti-edges can change. Changing the technical 
data as in the second operation, changes only one term $f(S)$ and by at most $k$. So it can be 
shown by induction using the recursive edge equation, \autoref{separation_edges}, that the edge 
number 
changes by at most $k$. 
So the above inequality remains true each time we  perform operation 2. 
\end{proof}

\begin{proof}[Proof of \autoref{saturn}.]
 Combine \autoref{is_astr} with \autoref{saturn1}.
\end{proof}

 \begin{obs}\label{normal_leq}
 Let $\widetilde T$ be a saturation of $\widehat T$. 
 Then $\widetilde T$ has at least as many normal atoms as $\widehat T$.
 \end{obs}
 
 \begin{proof}
  None of the three saturation operations changes normal atoms. 
 \end{proof}

 \section{From atoms to molecules}\label{sec:gen_case}
 
We will give an overview of the rest of the proof \autoref{abstr_sep_tr}. Above we have 
already proved \autoref{abstr_sep_tr} for normal saturated \astr s $\widehat T$. Our 
strategy is to deduce the general case from this special case. For that we would like to 
understand how much deleting a tiny atom changes the number $e(\widehat T)-\beta \cdot (n-k)$. It 
turns 
out 
that deleting some tiny atoms increases that number but for others it may decrease. Another thing 
that makes such an analysis more complex is that if we delete some tiny atom, it may change which 
branches are big and small, hence it might affect how much that number changes for other tiny atoms.

We overcome these difficulties as follows. We group the tiny atoms in `molecules' so that -- very 
roughly -- if the deletion of one of these atoms `affects' the other, then they are in the same 
molecule.
Then we analyse how $e(\widehat T)-\beta \cdot (n-k)$ changes if we delete a whole molecule. 

\vspace{.3cm}

\begin{rem}
 Throughout this section we will talk about vertices of tiny atoms. This is an abuse of notation as 
atoms of \astr s have no vertices but just a number of tiny vertices, compare 
\autoref{tiny_vertices}. It is straightforward to rephrase the corresponding places using numbers 
of tiny vertices instead of tiny vertices themselves. Although this is slightly informal, we 
believe that this makes it easier to read.
\end{rem}

\autoref{fig_desire_et_al} provides an example illustrating the following definitions. 
   \begin{figure} [htpb]   
\begin{center}
   	  \includegraphics[height=8cm]{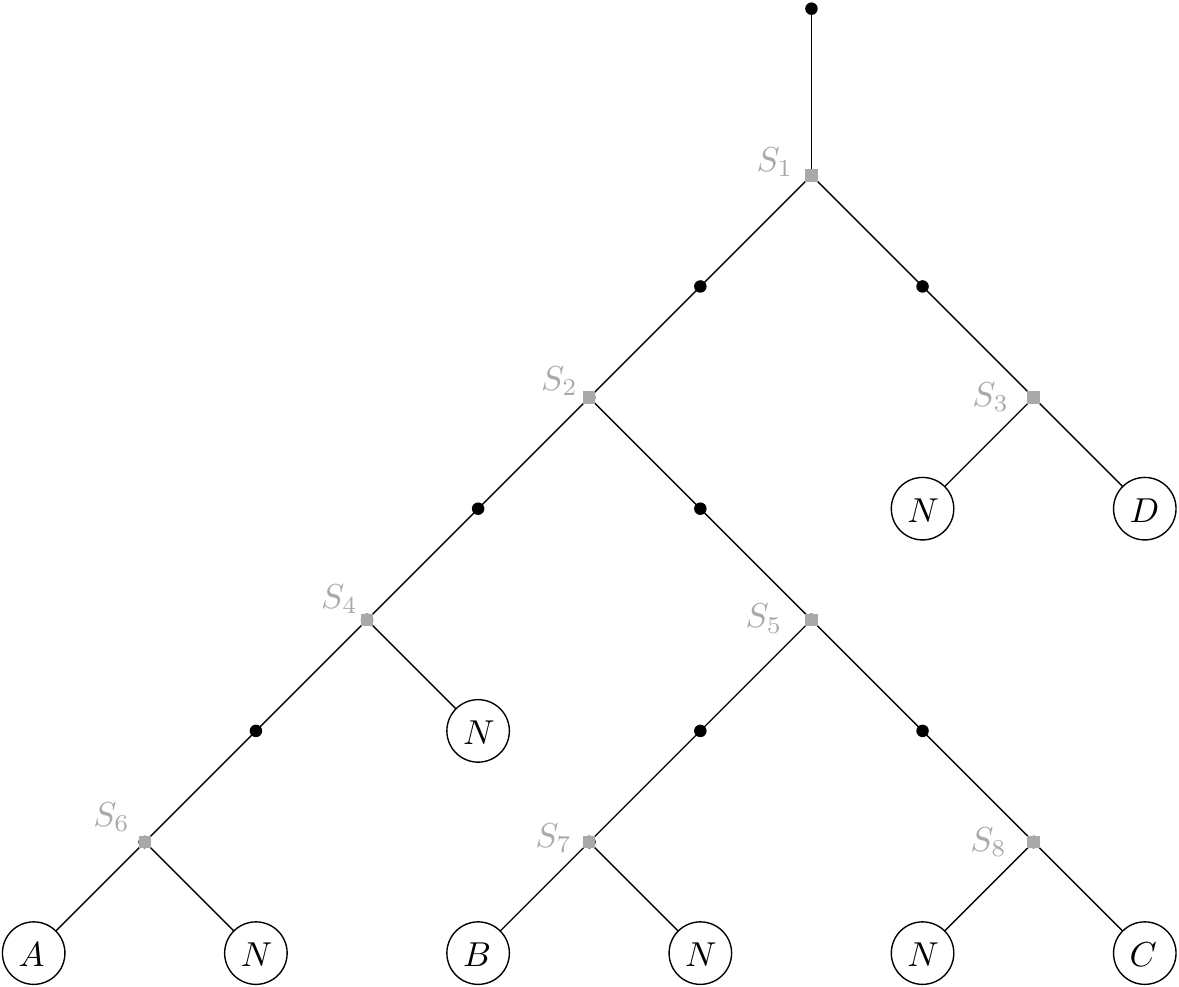}
   	  \caption{An abstract separator-tree. The separators are coloured in grey. The atoms are 
indicated by circles; normal atoms are labelled with the letter $N$. The tiny atoms are labelled 
by $A$, $B$, $C$ and $D$.
The separators $S_2$, $S_4$ and $S_5$ are normally balanced. The separator $S_1$ has only one 
branch 
that is almost small, this is the branch containing $D$. 
The atoms $B$ and $C$ both desire the separator $S_5$ but only one of them achieves it. The atom 
$A$ desires $S_4$ but does not achieve it. The atoms $A$, $B$ and $C$ form a molecule whose final 
separator is above $S_1$ (if existent). The atom $D$ achieves $S_1$ and forms a molecule by itself. 
Its final separator is $S_1$. 
   	  }\label{fig_desire_et_al}
\end{center}
   \end{figure}

Recall that a tiny atom $A$ \emph{desires} a separator $S$ if the branch of $S$ containing $A$ is 
almost small 
for $S$ and there is no separator below $S$ 
that has a small 
branch containing $A$ other than the separator just above $A$, and $S$ is not the separator just 
above $A$. A tiny atom \emph{achieves} a separator if it desires it and is contained in the small 
branch of that separator. Each tiny atom achieves at most one separator; if so, this is the highest 
separator 
it desires.

Two atoms \emph{affect} each other if they both desire a common separator. 
A  \emph{molecule} is a smallest set of atoms that is closed under the symmetric relation 
`affecting each other'. In other words, it is an equivalence class of the transitive closure of 
the symmetric relation `affecting each other'. 

The \emph{final separator} of a molecule is the highest separator achieved by any atom in the 
molecule; if some of its atoms do not achieve a separator, the final separator is undefined. 
The \emph{reach} of a molecule is the number of normal atoms in the small branch of its final 
separator; if the final separator is undefined, this number is infinite. We denote the reach of 
a molecule $M$ by 
$\ell_M$.

\begin{obs}\label{final_distinct}
No two molecules have the same final separator.
\end{obs}

\begin{proof}
 Any final separator is desired by one of its atoms. As no two atoms in different molecules desire 
the same separator, final separators have to be distinct for different molecules. 
\end{proof}

\begin{obs}\label{normally_balanced1}
 Any separator desired by an atom but not achieved by it is normally balanced.
\end{obs}

\begin{proof}
Let $A$ and $B$ be the two branches of a separator $S$. Assume an atom of the branch $A$ 
desires, the separator $S$. So the branch $A$ is almost small.  If that atom does not achieve $S$, 
an atom of the other branch achives, and hence desires the separator $S$. Thus also the branch $B$ 
is almost small. Thus the separator $S$ is normally balanced. 
\end{proof}

\begin{lem}\label{normally_balanced2}
 If a separator is achieved by an atom of a molecule, then it is normally balanced or the final 
separator of that molecule. 
\end{lem}

\begin{proof}
Let $S$ be a separator achieved by a tiny atom $A$. We assume that $S$ is not normally balanced and 
have the aim to show that $S$ is the final separator of the molecule $M$ of $A$. 

Let $A'$ be the branch of $S$ containing $A$. As the atom $A$ desires the separator $S$, it cannot 
be the separator just above $A$; that is, the branch $A'$ is not just a single tiny atom. Hence the 
separator $S$ cannot be the separator just above any tiny atom of $A'$. 
As the atom $A$ desires the separator $S$, the branch 
$A'$ is almost small for $S$. 
So any tiny atom of the branch $A'$ either 
desires $S$ or achieves a separator of $A'$.

As the separator $S$ is not normally balanced, the branch of $S$ different from $A'$ cannot be 
almost small. Hence no atom in that branch desires $S$.

Thus the molecule $M$ of $A$ has all its atoms in the 
branch $A'$. So the separator $S$ is the final separator of the molecule $M$. 
\end{proof}

\begin{lem}\label{final_not_normal}
 No final separator of a molecule is normally balanced.
\end{lem}

\begin{proof}
Suppose for a contradiction there is a final separator $S$ of a molecule that is normally balanced. 
Then $S$ is desired by some atom of $M$. As this atom is contained in the small branch of $S$, the 
big branch also has to contain an atom desiring $S$ as $S$ is normally balanced.
Then this atom must be in the molecule $M$. As not both these atoms can achieve $S$, one of these 
atoms does not achieve any separator or achieves a separator strictly above $S$. This is a 
contradiction to the assumption that $S$ is final for $M$.
\end{proof}

The \emph{reach of a separator} is the number of normal atoms in its small branch. 

\begin{lem}\label{separator_reach}
For a tiny atom $A$, let $S_1, ..., S_m$ be the separators it desires given in increasing order 
in the tree-order. Then the reach of $S_{i+1}$ is at least twice the reach of $S_i$. 
\end{lem}

\begin{proof}
By the definition of reach, the number of normal atoms below $S_i$ is at least twice the reach of 
$S_i$.
By the definition of desires, the branch of each $S_i$ containing $A$ is almost small for all 
$i$. Hence the reach of $S_{i+1}$ is equal to the number of normal atoms in the branch of 
$S_{i+1}$ containing $A$. As this branch contains all (normal) atoms below $S_i$, the lemma 
follows. 
\end{proof}

We say that two atoms are \emph{equivalent} if 
they achieve the same separator. 
\begin{rem}
Atoms that do not achieve any separator are not in any equivalence class. We do not have to include 
such atoms into our analysis, expressed by the term `equivalent', as for such atoms the 
inequalities we estimate below will turn out to be trivially satisfied.
\end{rem}
Note that any two equivalent atoms are in the same molecule. 
Clearly, the set of separators desired by a single atom forms a chain in the tree-order. This 
property also holds for equivalence classes:

\begin{lem}\label{equivi}
 The set of separators desired by atoms in an equivalence class forms a chain in the tree-order.
\end{lem}

\begin{proof}
 Let $A$ and $B$ be two atoms in the same equivalence class; that is, they achieve the same 
separator $S$. Suppose for a contradiction that $A$ and $B$ desire separators $S_1$ and $S_2$ that 
are incomparable in the tree-order. Then there is a separator $S'$ below or equal to $S$ such that 
$S_1$ and $S_2$ are in different branches of $S'$. So $A$ and $B$ are in different branches of 
$S'$. The separators $S_1$ and $S_2$ witness that $S'$ is not the separator just above $A$ or $B$. 
Hence one of $A$ or $B$ achieves the separator $S'$. -- But only one of them can achieve it as they 
are contained in different branches. This is a contradiction to the assumption that they achieve 
the same separator. 
Hence the separators desired by atoms in an equivalence class forms a chain in the tree-order.
\end{proof}

We say that a tiny vertex \emph{is in} an equivalence class if it is in an atom of that equivalence 
class. 
The \emph{reach} of a tiny atom is reach of the separator it achieves; if it does not achieve 
any 
separator, its reach is infinite. 
The \emph{reach} of a tiny vertex is the reach of the tiny atom containing it.  We denote the reach 
of a vertex $v$ by $\ell_v$ and of an atom $A$ by $\ell_A$. 
\begin{rem}
 The reach of vertices in a molecule is a key quantity in our estimate below of how it would effect 
the number of edges if we deleted the molecule. 
\end{rem}

The next lemma, see \autoref{fig_uni}, will be applied to estimate the numbers $\ell_v$ and 
$\ell_A$.

   \begin{figure} [htpb]   
\begin{center}
   	  \includegraphics[height=4cm]{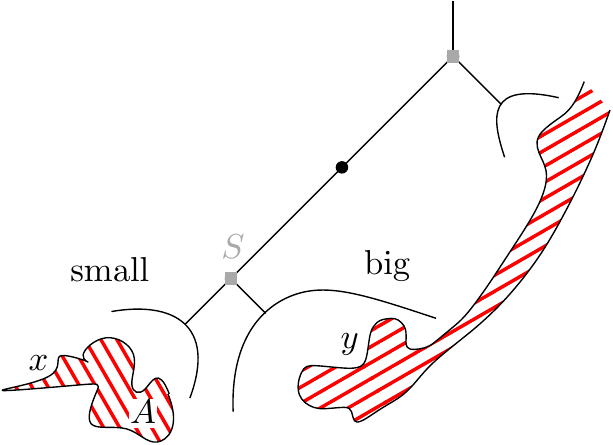}
   	  \caption{The situation of \autoref{double_universal}. The separator $S$ is
depicted in grey. The equivalence class $x$ is contained in its small branch, the 
equivalence class $y$ intersects the big one. }\label{fig_uni}
\end{center}
   \end{figure}
\begin{lem}\label{double_universal}
Let $A$ be a tiny atom of a molecule $M$ with reach $i< \ell_M$. 
 Let $S$ be the separator achieved by $A$. 
  Then all atoms of the equivalence class $x$ of $A$ are contained in the small branch of $S$.
 There is a unique equivalence class $y=y(A)$ that contains atoms of the big branch 
of $S$ and that has reach at least $2i$. Moreover:
  \begin{itemize}
  \item The equivalence class $x$ has at most as many tiny vertices as $y$.
  \item For atoms $A$ and $A'$ in the same molecule $M$ but in different equivalence classes whose 
reach is less than $\ell_M$ but within a factor of two, the equivalence classes $y(A)$ and 
$y(A')$ disagree.
 \end{itemize}
\end{lem}

\begin{proof}
As the atom $A$ achieves the separator $S$, no atom of the other branch can achieve the 
separator $S$. Hence all atoms of the equivalence class $x$ of the atom $A$ are contained in the 
small branch of the separator $S$. As the reach $i$ of the atom $A$ is less than $ \ell_M$, the 
separator $S$ is not final for the molecule $M$ of $A$. Hence by \autoref{normally_balanced2}, the 
separator $S$ is normally balanced. So the big branch $B$ of $S$ must contain an atom desiring $S$. 
Let $y$ be its equivalence class. By \autoref{separator_reach}, the atoms of the equivalence 
class $y$ have reach at least $2i$. 

It follows from the fact that the separator $S$ is normally balanced that the equivalence class $x$ 
has at most as many tiny vertices as $y$.

Now let $A$ and $A'$ be atoms of $M$ in different equivalence classes whose reach is less than 
$\ell_M$ 
but 
within a factor of two. By \autoref{equivi}, the set of separators desired by atoms in the 
equivalence class $y(A)$ forms a chain. By \autoref{normally_balanced1} all separators in this 
chain 
except for the last one are normally balanced. As shown above, the separator $S$ achieved by the 
atom $A$ is a separator of this chain except for the last one. The separator $S'$ achieved by the 
atom $A'$ is not equal to $S$ by definition of equivalence class and it cannot be equal to any 
other separator in this chain as its reach is within a factor of two of the reach of the separator 
$S$. Hence the separator $S'$ desired by an atom of the equivalence class $y(A')$ is not desired by 
any atom in the equivalence class $y(A)$. In particular, the equivalence classes $y(A)$ and $y(A')$ 
are distinct. 
\end{proof}

In a slight abuse of notation, we say that a tiny vertex \emph{is in a molecule} if its tiny atom 
is in that molecule.

\begin{cor}\label{many_desired_vs_better}
In any molecule $M$, amongst its vertices with reach at least $i$, at least 
half of them have reach at least $2i$ for all integers $i$ with $1\leq i \leq \ell_M/2$. 
\end{cor}

\begin{proof}
 By $X[i,2i)$ we denote the set of equivalence classes whose atoms have reach at least $i$ 
but less than $2i$. Similarly, by $X[2i)$ we denote the set of equivalence classes whose atoms 
have reach at least $2i$.

Let $A$ be a tiny atom. Then the equivalence class $y(A)$ defined in \autoref{double_universal} is 
the same for any atom equivalent to the atom $A$. Hence this defines a function from the set 
$X[i,2i)$
of equivalence classes into the set $X[2i)$. By \autoref{double_universal}, this function is 
injective, and the image of every $x\in X[i,2i)$ has at least as many tiny vertices as $x$. 
So the number of tiny vertices of equivalence classes in $X[i,2i)$ is at most the number of tiny 
vertices with equivalence classes in $X[2i)$. This proves the lemma. 
\end{proof}

Given a molecule $M$, we denote by $X_i$ the set of those equivalence classes of $M$ 
the 
reach 
of whose atoms is in the half-open interval $[\frac{\ell_M}{2^{i+1}},\frac{\ell_M}{2^{i}} )$. 
\begin{obs}\label{x01}
 $|X_0|=1$. 
\end{obs}
\begin{proof}
 Any molecule has only one final separator, and all atoms achieving that separator are in the same 
equivalence class. All values of reaches of atoms in a molecule except 
$\ell_M$ are at most half that value.
\end{proof}

\begin{cor}\label{many_desired_xi_prep}
For any molecule $M$, we have $|X_{i+1}|\leq \sum_{j\leq i}|X_j|$ for all integers $i$ with $0\leq 
i \leq 
\log_2(\ell_M)-1$. 
\end{cor}

\begin{proof}
Let $A$ be a tiny atom. Then the equivalence class $y(A)$ defined in \autoref{double_universal} is 
the same for any atom equivalent to the atom $A$. Hence this defines a function from the set $X_i$ 
of equivalence classes into the union of the 
sets $X_j$ with $j\leq i$. By \autoref{double_universal}, this function is injective. This implies 
that $|X_{i+1}|\leq \sum_{j\leq i}|X_j|$.  
\end{proof}

\begin{rem}
 It can be shown that the injection in the above proof is in fact a bijection. We will not use this 
fact. 
\end{rem}

\begin{cor}\label{many_desired_xi}
For any molecule $M$, we have $|X_{i}|\leq 2^{i-1}$ for all integers $i$ with $1\leq i \leq 
\log_2(\ell_M)-1$. 
\end{cor}

\begin{proof}
Recall that $|X_0|=1$ by \autoref{x01}. The statement $|X_{i}|\leq 2^{i-1}$ follows by strong 
induction on $i$ from  
\autoref{many_desired_xi_prep}. The first instance of the induction is that $|X_1|=1$. 
\end{proof}

\autoref{many_desired_vs_better} can be used to prove the following estimate for the sum of the 
reciprocates of the reaches over all vertices in a molecule. 
The \emph{vertex number} $n_M$ of a molecule $M$ is the number of its tiny vertices.

\begin{lem}\label{reach_sum}
For any molecule $M$, the sum $\sum \frac{1}{\ell_v}$ over all vertices in the molecule is 
upper-bounded by 
\[
 \frac{2}{3} \cdot n_M+  \frac{1}{3 \ell_M^2}\cdot n_M
\]

\end{lem}

\begin{proof}
In order to estimate the sum  $\sum \frac{1}{\ell_v}$ we estimate the expected value of 
$\frac{1}{\ell_v}$ where we pick a tiny vertex $v$ uniformly at random from the molecule $M$. 
First we do this by consciously ignoring the assumption in \autoref{many_desired_vs_better} that 
$i$ is upper-bounded by $\ell_M/2$. 
At most half vertices $v$ have reach one, amongst the remaining at most half have reach two, 
amongst the remaining at least half have reach four etc.
So the expected value is upper-bounded by
\[
 E\left (\frac{1}{\ell_v}\right)\leq \frac{1}{2}\left( 1+\frac{1}{4} + \frac{1}{4^2}+...  \right)  
 \]

So the expected value $E(\frac{1}{\ell_v})$ is upper-bounded $\frac{2}{3}$ -- when ignoring the 
bound on $i$ in 
\autoref{many_desired_vs_better}. If we take the bound that $1\leq i \leq \ell_M/2$ into account, 
we 
have to add an error 
term 
as we cannot apply \autoref{many_desired_vs_better} for too large values of $i$.
So now we repeat the upper argument but just assume that all vertices not mentioned in the first 
$\log_2(\ell_M)-1$ steps have reach at least $\ell_M$. The probability mass of them is 
$\frac{1}{2^{y}}$, where we abbreviate $y=\log_2(\ell_M)$. So when we take the bound on $i$ into 
account, we modify the above estimate by replacing the sum of all but the first $y-1$ terms 
by the term $\frac{1}{\ell_M} \cdot \frac{1}{2^{y}}$. 
This gives that $ E(\frac{1}{\ell_v})\leq \frac{2}{3}+X$, where we estimate the error term $X$ as 
follows\footnote{The first term of the infinite sum below is the correct one. Indeed, this infinite 
sum ranges over those summands from the old sum with $\ell_v$ is at least $\ell_M$. 
These are precisely the ones where we can no longer apply \autoref{many_desired_vs_better}.}.
\[
X\leq  \frac{1}{\ell_M} \cdot \frac{1}{2^{y}}\cdot - \frac{1}{2}  \left(   
\frac{1}{4^y} + 
\frac{1}{4^{y+1}}+ ... \right)
\]
\begin{comment}
 I checked this in detail and the indices are correct!
\end{comment}

This simplifies to: 
\[
 X\leq \frac{1}{3}\cdot \frac{1}{\ell_M^2} 
\]
So we get:
\[
 E\left (\frac{1}{\ell_v}\right)\leq \frac{2}{3} + \frac{1}{3}\cdot \frac{1}{\ell_M^2} 
\]
By multiplying both sides with $n_M$ one gets the statement of the lemma. 
\end{proof}

Given an \astr\ $\widehat P$ with a molecule $M$, then $\widehat P-M$ is obtained from $\widehat P$ 
by deleting all tiny vertices in $M$. By the definition of a molecule $\widehat P-M$ is an \astr\ 
with 
the same choices of big and small branches at the separators.
The \emph{edge number} $e(M)$ of $M$ is defined to be $e(\widehat P)-e(\widehat P-M)$.
Our aim is to prove the following.

\begin{lem}\label{molecule_removal}
For any molecule $M$ of a saturated \astr, its edge number is upper-bounded by 
\[
 e(M)\leq \beta \cdot n_M+  \frac{k^2\cdot \log_2(\ell_M)}{18 \cdot\ell_M^3}
\]
\end{lem}

\begin{rem}
If there is a saturated \astr\ $\widehat P$ with only a single molecule $M$ whose reach $\ell_M$ is 
infinite, 
then \autoref{molecule_removal} combined with \autoref{voll_normal} applied to $\widehat P-M$ 
immediately implies \autoref{abstr_sep_tr} for $\widehat P$. The proof in the general 
case 
will be similar. 
\end{rem}

\begin{proof}[Proof of \autoref{molecule_removal}.]
First we make the simplifying assumption that every tiny vertex is assigned to the separator 
achieved by its atom if existent. Later we will show how the general case can be reduced to this 
special case.
 
Recall that two atoms in a molecule are equivalent if they achieve the same separator. 
Now we will estimate by how much $e(\widehat P)$ changes if we delete all tiny vertices whose atoms 
are 
in a single equivalence class. For now we ignore that the structure after deletion might no longer 
be an \astr\ as the condition for the small and big branches might not be satisfied. This will not 
be a problem at the end as we will delete a whole molecule. Still for the analysis it is easier to 
focus on a single equivalence class.

Now we fix an equivalence class $Y$ and let $A_1,...,A_i$ be its atoms. By $m_i$ we denote the 
number of tiny vertices in $A_i$ and by $m$ we denote the sum of the $m_i$. 

Throughout this proof we will use several times that all separators have vertex number $k$ as 
the \astr\ is saturated by assumption. 
By deleting the tiny vertices of $A_i$ (formally that means, we replace the atom $A_i$ by an atom 
of size $k$, and set in the technical data all terms $m(A_i,\bullet)$ equal to zero.) we lose at 
most $k\cdot m_i+\frac{m_i(m_i-1)}{2}$ edges in the atom $A_i$.
If some (or equivalently: any) atom of the equivalence class does not achieve any separator, we do 
not lose any further edges.
Otherwise, the only other edges we lose are in the separator achieved by the atoms in the 
equivalence class. The free anti-edges after deletion are\footnote{Recall that $\ell_A$ denotes the 
number of normal atoms in the branch $A$. } $\frac{k^2}{2}- 
\frac{k^2}{2\ell_A}$. Before the deletion they are 
$\frac{k^2}{2}-\frac{(k-m)^2}{2\ell_A}-\sum_i\frac{m_i(m_i-1)}{2}-\frac{m}{2}$. 
This term minus the previous one measures by how much the number of free anti-edges goes down after 
deletion. This difference evaluates to:
\[
  \frac{k\cdot m}{\ell_A}  -   \frac{m^2}{2\ell_A} -             \sum_i\frac{m_i(m_i-1)}{2} - 
\frac{m}{2}
\]
As free anti-edges increase the edge number, the total sum of the edges we lose when deleting an 
equivalence class is:
\[
 km+ \frac{k\cdot m}{\ell_A}  -   \frac{m^2}{2\ell_A} - \frac{m}{2}
\]
So this is per tiny vertex $v$: 
\begin{equation}\label{relative}
 k+\frac{k}{\ell_v}- \frac{m}{2\ell_v}-\frac{1}{2}
\end{equation}
  
Now we want to compare these values for different equivalence classes of a molecule. 
So we no longer fix an equivalence class $Y$. 
We recall that the reach $\ell_A$ is constant for all atoms $A$ in an equivalence class $Y$; we 
denote that value by $\ell_Y$ .
The only term in \autoref{relative} except `$\ell_Y$' that depends on 
the equivalence class is `$m$'. Below we will refer to that value by writing `$m_Y$' instead of 
`$m$'.

\begin{sublem}\label{equivalence_classer}
 The sum $Z$ over $\frac{m_Y^2}{\ell_Y}$ over all equivalence classes $Y$ of the molecule $M$ is 
lower-bounded by $\frac{n_M^2}{\ell_M\cdot \log_2(\ell_M)}$.
\end{sublem}
\begin{proof}
We partition the set of equivalence classes of the molecule $M$ into sets $X_i$, where $i$ is an 
integer between zero and 
$\log_2(\ell_M)-1$. We put an equivalence class $Y$ in the set $X_i$ if its reach $\ell_Y$ is in 
the 
half-open interval $[\frac{\ell_M}{2^{i+1}},\frac{\ell_M}{2^{i}} )$. We recall that $|X_0|=1$ by 
\autoref{x01}.

By \autoref{many_desired_xi}, $|X_{i}|\leq 2^{i-1}$ for all $i$ with $1\leq i\leq 
\log_2(\ell_M)-1$. 
Below, slightly generously, we will only use that $|X_{i}|\leq 2^{i}$; but that inequality holds 
additionally for the value $i=0$. 

Now we split the sum over all equivalence classes into separate sums, one for each $X_i$.
We estimate using Cauchy-Schwartz 
\[
 \sum_{Y\in X_i} \frac{m_Y^2}{\ell_Y}\geq \frac{1}{\ell_i\cdot |X_i|} x_i^2;
\]
here we abbreviate $x_i=   \sum_{Y\in X_i} m_Y$ and $\ell_i$ is equal to $\ell_Y$ for all $Y\in 
X_i$. Now we use the fact that $\ell_i\leq 
\frac{\ell_M}{2^{i}}$ and $|X_i|\leq 2^i$:. So:
\[
 \sum_{Y\in X_i} \frac{m_Y^2}{\ell_Y}\geq \frac{1}{\ell_M} x_i^2;
\]

So the sum $Z$ is lower-bounded by
\[
 Z\geq \sum_{i=0}^{\log_2(\ell_M)-1}   \frac{1}{\ell_M} x_i^2
\]
We apply Cauchy-Schwarz again and use the fact that the sum of the $x_i$ is equal to the 
number $n_M$ of tiny vertices of the molecule $M$ to deduce that
\begin{equation}\label{eqZ}
  Z\geq \frac{n_M^2}{\ell_M \cdot \log_2(\ell_M)}
\end{equation}

\end{proof}

Now we combine \autoref{equivalence_classer} with \autoref{reach_sum} to estimate via
\autoref{relative} the number $e(M)$ of edges we lose when we delete the whole 
molecule\footnote{The second summand comes from the fact that in \autoref{relative} the last 
summand is minus a half.}. 
\begin{equation}\label{mole}
 e(M)\leq\frac{5}{3} k \cdot n_M-\frac{1}{2} n_M+  \frac{k}{3 \ell_M^2}\cdot n_M- 
\frac{n_M^2}{2\ell_M\cdot 
\log_2(\ell_M)}
\end{equation}
The sum of the first two terms on the right hand side is equal to $\beta \cdot 
n_M$. 
We abbreviate the right hand side without these two summands by $D$. It remains to show that 
$D$ 
is at most the last term appearing in the statement of this lemma. For that we consider 
$D$ as a function of $n_M$ and compute the first 
derivative. This yields that $D$ is maximal if $n_M= \frac{k \cdot \log_2(\ell_M)}{3 
\cdot \ell_M}$. In this case $D$ evaluates to:
\begin{equation}\label{Deq} 
 D\leq \frac{k}{6 \ell_M^2}\cdot n_M = \frac{k^2\cdot \log_2(\ell_M)}{18 \cdot\ell_M^3}
\end{equation}

 This completes the proof in the special case that every tiny vertex is assigned to the separator 
achieved by its atom. Next we reduce the general case to this special case. 

 \vspace{.3cm}
 
If a vertex $v$ is assigned to a separator $S'$ different from the separator $S$ achieved by 
its atom that separator must be above $S$ in the \astr. As the vertex $v$ is in the small branch of 
the separator $S$, the 
number of normal atoms below the small branch of $S'$ is at least two times the reach $\ell_v$ of 
$v$. 

For a tiny vertex $v$ of the molecule $M$, we denote by $\bar \ell_v$ the number of normal atoms in 
the small branch of the separator $v$ 
is assigned to; if $v$ is not assigned to any separator this number is infinite. We have shown that 
if $\bar \ell_v$ is not equal to the reach $\ell_v$, then it is at least twice as big.

We denote by $W$ the set of vertices $v$, where the reach $\ell_v$ is different from $\bar 
\ell_v$. Each vertex $v$ in $W$ makes the first term of \autoref{mole} smaller by at least 
$\frac{k}{\ell_v}- \frac{k}{\overline{\ell_v}}\geq \frac{1}{2}\frac{k}{\ell_v}$, which we 
generously 
estimate by $\frac{1}{2}\frac{k}{\ell_M}$. Moreover, the fact that the vertex $v$ is 
assigned to a different separator does not affect the second term but it may influence the last 
term. 

We refer to the number $\frac{1}{2}\frac{k}{\ell_M}$ as the \emph{$W$-defect}.
It remains to show that the influence on the last term is at most the $W$-defect per 
vertex in $W$. The average contribution of a 
vertex of $M$ to the last term is:
\[
\frac{n_M}{2\ell_M\cdot \log_2(\ell_M)}
\]
First, as a heuristic, we note that if we plug in the value of $n_M$ computed above, this is much 
less than the $W$-defect. This heuristic suggests that it is very 
likely that $e(M)$ is small if the set $W$ is large. 

Let us now be precise. 
We argue as above in the estimate for the last term and get the same bound except that we 
replace `$n_M$' by `$(n_M-w)$', where $w=|W|$. So we get the same inequality as in \autoref{mole} 
but with the term $D$ replaced by the following.
\[
D'=  \frac{k}{3 \ell_M^2}\cdot n_M -  \frac{(n_M-w)^2}{2\ell_M\cdot 
\log_2(\ell_M)} - \frac{1}{2}\frac{k}{\ell_M}\cdot w
\]
We consider $D$ and $D'$ as functions of the variable $n_M$. Comparing the first and last term of 
$D'$ yields that $D'(n_M)\leq D(n_M-w)$. Hence the above estimate for $D$, \autoref{Deq}, also 
holds with $D'$ in place of $D$. So 
we get the same estimate in the general case.  
\end{proof}

We summarise this section in the following. 

\begin{proof}[Proof of \autoref{abstr_sep_tr}.]
Let $G$ be a graph with $n\geq 2k+1$ vertices without a $(k+1)$-connected subgraph with more than 
$2k$ vertices. By \autoref{sep--tr_exist}, the graph $G$ has a separator-tree $T$. Let 
$\widehat T'$ be its 
associated framework (as defined above \autoref{rem43}), and let $\widehat T$ be its 
associated \astr\ (as defined above \autoref{eg44}). 
By \autoref{eg44} the edge numbers of $\widehat T'$ and $\widehat T$ agree; here we will show that 
the later one is bounded by the bound given in \autoref{abstr_sep_tr}. 
Let $\widetilde T$ be a saturation of 
$\widehat 
T$. By \autoref{is_astr}, $\widetilde T$ is a saturated \astr. 

\begin{sublem}\label{dege_case}
 If $\widetilde T$ has no normal atom, then 
 \[
  e(\widehat T)\leq \beta \cdot (n-k)+ \frac{\gamma}{L}-\epsilon
 \]
\end{sublem}

\begin{proof}
By \autoref{normal_leq}, the \astr\ $\widehat T$ has no normal atom; so $L=1$. 
We denote the vertex number of $\widetilde T$ by $\tilde n$. 
As the \astr\ $\widetilde T$ is saturated but has no normal atom, it has only one node. 
So the edge number of $\widetilde T$ is $\frac{\tilde n(\tilde n-1)}{2}$. 
And the vertex number $\tilde n$ is not more than $k+\frac{k}{3}$. 
By \autoref{saturn1} we get that the edge number $e(\widehat T)$ is upper-bounded by
$e(\widetilde T)+k (n-\tilde n)$. So it remains to show the following.
\[
 \frac{\tilde n(\tilde n-1)}{2}+ k (n-\tilde n)\leq \beta \cdot (n-k)+\gamma-\epsilon
\]
By considering the first derivative, we conclude that the left hand side takes its minimum at 
$\tilde n=k+\frac{1}{2}$. As it is a quadratic polynomial in $\tilde n$, it is symmetric around 
that value and takes its maximum at the boundary of the interval. As the boundary point one is 
farer away from $k+\frac{1}{2}$ than $k+\frac{k}{3}$, it takes its maximum at one (if $k>1$; and 
if $k=1$ the integer $\tilde n$ must be one as the upper-bound is less than two). So it suffices to 
check this 
inequality for $\tilde n=1$. We also note that the coefficient in front of $n$ is larger on the 
right than on the left. Hence we may assume that $n$ is minimal. That is, $n=2k+1$. 
Then it evaluates to:
\[
 k \cdot 2k\leq \beta \cdot (k+1)+\gamma-\epsilon
\]
Plugging in the values for $\beta$ and $\gamma$ and $\epsilon$ proves this inequality. 
\end{proof}

We obtain $T'$ from  $\widetilde T$ by deleting all tiny vertices in molecules. By the 
definition of 
molecules, $T'$ is an \astr. Moreover all of its atoms are normal. 
By \autoref{dege_case} we may assume that the tree $\widetilde T$ has at least one normal atom, 
thus the tree $T'$ is 
nonempty. 
Hence we can apply 
\autoref{voll_normal} to the abstract separator-tree $T'$. So the edge 
number $e(T')$ is 
upper-bounded by $\beta 
\cdot (n(T')-k)+ \frac{\gamma}{\widetilde L}-\epsilon$ 
minus $\gamma$ times 
the branching error sum of $\widehat T$; here $\widetilde L$ denotes the number of normal atoms of 
$T'$, which is 
equal to the number of normal atoms of $\widetilde T$.

 The next step is to extend that estimate from the \astr\ $T'$ to the \astr\ $\widetilde T$ using 
\autoref{molecule_removal}. 
For this to work we have to show that the molecule error term, this is the last 
summand in the inequality of that lemma, is compensated by some branching error term. Indeed, we 
will show that the branching error term at the final separator of a molecule $M$ is never smaller 
than the molecule error term of $M$. By \autoref{final_not_normal} the final separator of a 
molecule is not normally balanced. So by \autoref{calci} for that it suffices to show that 
\[
 \frac{k^2}{3} \cdot \frac{1}{4\ell_M^2}  \geq \frac{k^2\cdot \log_2(\ell_M)}{18 
\cdot\ell_M^3};
\]
which is obviously true as $\ell_M\geq 1$. 

Since by \autoref{final_distinct} no two molecules have the same final separator, we can indeed 
compensate the sum of all molecule error terms by the branching error sum. Hence the edge 
number $e(\widetilde T)$ is 
upper-bounded by $\beta \cdot (n_{\widetilde T}-k)+ \frac{\gamma}{\widetilde L}-\epsilon$. 

Finally by \autoref{saturn1} the edge number $e(\widehat T)$ is upper-bounded by \newline {$\beta 
\cdot 
(n-k)+ \frac{\gamma}{\widetilde L}-\epsilon$}. By \autoref{normal_leq}, the number $\widetilde L$ 
of 
normal atoms of 
$\widetilde T$ is at least the number of normal atoms of $\widehat T$; and $\widetilde L$ is at 
least one in this case.
This completes the proof. 
 \end{proof}

At the beginning of \autoref{sec:normal_atoms} we have shown that 
\autoref{abstr_sep_tr} implies 
\autoref{main-intro}. So this completes the proof of \autoref{main-intro}.

\section{Extremal constructions and lower order terms}\label{extremal}

In this section, we construct a family of graphs that shows that the constant in 
\autoref{main-intro} is best-possible. We also state \autoref{detailed} below which is an 
improvement of \autoref{main-intro} in lower order terms. The bounds in this conjecture are 
attained 
by the constructed family. 

\begin{eg}\label{our_extremal}
For powers $2^i$ of two we build graphs $G_i$ that have separator-trees with $2^i$ normal 
atoms that attains the bounds in \autoref{voll_normal} and hence in \autoref{abstr_sep_tr}.  
For $i=0$, the graph $G_0$ is a complete graph on $2k$ vertices. 
We obtain $G_{i+1}$ from two disjoint copies $A$ and $B$ of $G_i$ as follows.
We pick a separator $S$ of size $k$ in each of $A$ and $B$ and glue the two graphs $A$ and $B$ at 
that separator. We pick the separator $S$ in $A$ so that it takes $\frac{k}{2^i}$ vertices from 
each of the $2^i$ atoms of $A$. These vertices can be picked disjoint from all previously picked 
separators as each atom has $2k$ vertices. For $B$ we do the same. It is straightforward 
to check that the graphs $G_i$ for $2^i\leq k$ attain equality in all inequalities in the proof of 
\autoref{voll_normal}. 
\end{eg}

\begin{rem}
For $m> 2k$, one could do the same construction, where $G_0$ is a complete graph on $m$ 
vertices. Similarly like \autoref{abstr_sep_tr} one can 
show that these graphs are edge-maximum without a $(k+1)$-connected subgraph with more than $m$ 
vertices amongst all graphs with the same number of vertices.
\end{rem}

\begin{eg}\label{our_extremal_huge}
The construction of \autoref{our_extremal} only makes sense if $2^i\leq k$. 
Here we continue this construction for large values of $i$. 
 For $k=2^i$ the graph $G^*=G_i$ has an independent set of $k$ vertices. As mentioned in 
\autoref{our_extremal}, its edge number is $\beta \cdot (n-k)+ 
\frac{\gamma}{L}-\epsilon$. Plugging in $n=k^2+k$ and $L=k$ yields:
\[
 e(G^*)=  \left(  \frac{5}{3}k-\frac{1}{2}-\frac{1}{6k}\right) (n-k)
\]
We construct larger graphs by gluing together several copies of the graph $G^*$ at the 
independent set of size $k$. So by \autoref{separation_edges} the edge number of these graphs is 
also described by the linear formula above. 
\end{eg}

We expect very much that the construction given in \autoref{our_extremal_huge} is 
extremal. More precisely, we are convinced that the methods of this paper also give a proof of 
the following (with a slightly more careful analysis). 
However, the details are technical and we will not give them here. 

\begin{con}\label{detailed}
Let $k\geq 1$ be a natural number. Every graph $G$ with $n\geq k^2+k$ vertices that has more 
edges than 
$\beta'' \cdot (n-k)$ has a $(k+1)$-connected subgraph 
with more than $2k$ vertices; here $\beta''=\frac{5}{3}k-\frac{1}{2}-\frac{1}{6k}$.
\end{con}

\section{Concluding remarks}\label{concl_rem}

We conclude with open problems and some remarks on Mader's conjecture. 
\begin{que}
Can you extend \autoref{main-intro} to hypergraph?
\end{que}
A notion that is similar to $(k+1)$-connected subgraphs is that of $(k+1)$-blocks. We refer the 
reader to \cite{zbMATH06427103}, where the question is asked which average degree forces a 
$(k+1)$-block.

\vspace{0.3cm}

In 1979 Mader conjectured that the following family is asymptotically extremal without a 
$(k+1)$-connected subgraph \cite{Mader_survey}. This family is obtained from an 
independent set $I$ of $k$ vertices by 
adding an arbitrary number of disjoint 
complete graphs on $k$ vertices completely to the vertex set $I$. 
The edge number of these graphs is $\frac{n-k}{k} \cdot \left(\frac{k(k-1)}{2} + k^2     \right)$.
This evaluates to $\frac{3k-1}{2}(n-k)$.

It is fairly easy to deduce from \autoref{abstr_sep_tr} Mader's conjecture  in the special case 
of graphs that have separator-trees with disjoint separators all of whose atoms have $2k$ vertices.
\begin{comment}
 Indeed, the argument for graph without $(k+1)$-connected subgraphs for at most $2k$ vertices gives 
that each atom $A$ has atomic defect $\alpha(A)\geq \frac{(n-k)^2}{3}$. Going through the proof, we 
throw this defect away the last time an anti-edge is in a separator. As the separators are 
disjoint, each anti-edge is at most twice in the separator. Hence we get at most half of it back. 
Doing the computations gives the result.  
\end{comment}
\begin{que}
Can you use the methods of this paper to improve the bound of Bernshteyn and Kostochka
on Mader's conjecture?
\end{que}

\vspace{.3cm}

The following construction has asymptotically the same number of edges as Mader's construction but 
has more edges for small graphs. 

\begin{eg}\label{mader_extremal}
This construction is similar to \autoref{our_extremal}  but we start with a different graph than 
$G_0$. Given a number $k$ that is a power of two, we define the graph $H_0$ to be the complement 
on $2k$ vertices of the disjoint union of the complete bipartite graphs $K_{m,m}$, where 
$m=\frac{k}{2^j}$ and $1\leq j\leq \log_2(k)$. 
\begin{comment}
We remark that the union of the graphs $K_{m,m}$ has 
$2k$ vertices-1.
\end{comment}
 We obtain the graph $H_{i+1}$ from two disjoint copies of the graph $H_i$ by gluing 
them together at a separator $S_i$ of $k$ vertices. Each graph $H_i$ is built up from $2^i$ copies 
of the graph $H_0$. In the graph $H_i$ the vertex sets of the graphs $K_{m,m}$ with 
$m=\frac{k}{2^{i+1}}$ are disjoint. We take their union to be the vertex set of the separator 
$S_i$. Clearly this union has $k$ vertices. 
The vertex number $n_i$ of the graph $H_i$ is $k+k \cdot 2^i$. 

Now we compute the edge number of the graph $H_i$. 
The edge number $e(H_0)$ of the graph $H_0$ is \[
             e(H_0)=      \frac{2k (2k-1)}{2} - \left(\frac{k}{2}\right)^2 -    
\left(\frac{k}{4}\right)^2 -...  - \left(\frac{k}{k}\right)^2              
                                      \] 
This evaluates to:
\[
 e(H_0)=\frac{5}{3} k^2 -k+\frac{1}{3}
\]
The edge number $e(S_i)$ of the separator $S_i$ is: 
\[
 e(S_i)= 2^i\cdot 2\cdot \frac{\frac{k}{2^{i+1}}\cdot \left(\frac{k}{2^{i+1}}-1\right)}{2}
\]
This evaluates to:
\[
 e(S_i)= \frac{k^2}{2^{i+2}}-\frac{k}{2}
\]
We shall prove inductively that
\begin{equation}\label{recursive}
  e(H_i)= \frac{3k-1}{2} (n_i-k) +\frac{n_i-k}{3k}+\frac{k^2}{3}\cdot 
\frac{k}{2(n_i-k)} -\frac{k}{2}
\end{equation}

The induction start is that $i=0$. \autoref{separation_edges} yields the recursive 
formula $e(H_{i+1})=2e(H_i)-e(S_i)$. This we use in the induction step:
\[
e(H_{i+1}) =  \frac{3k-1}{2} (n_{i+1}-k)+\frac{n_{i+1}-k}{3k}+ 2 \frac{k^2}{3}\cdot 
\frac{k}{2(n_i-k)} -k - \frac{k^3}{4(n_i-k)}+\frac{k}{2}
\]
It is straightforward to check that the right hand side of this equation is identical to the right 
hand side of \autoref{recursive} for the index $i+1$ in place of the index $i$. This completes the 
induction step. 

Finally, the graph $H_i$ for $i=\log_2(k)$ has $n_i=k+k^2$ vertices and  $\frac{3k-1}{2}(n-k)$ 
edges by \autoref{recursive}. This graph has an independent vertex set of size $k$. Hence we can 
glue any number of copies together at the set $I$ and still have the edge number 
$\frac{3k-1}{2}(n-k)$.
\end{eg}

We believe that the point where Mader's family and the family $H_i$ have the same edge number is 
the point where Mader's family begins to be best possible (together with other constructions such 
as the one described in the last line of \autoref{mader_extremal}). 
Before that point we believe that the family $H_i$ is best possible. 
And we believe that the bound for the graphs $H_i$ gives the correct bound for small graphs. 
For very small graphs with at most $2k$ vertices, it is easy to see that constructions like that of 
the graph 
$H_0$ above give the optimal bound, hence we exclude this easy case below. 
We summarise this in the following conjecture. 

\begin{con}\label{precise_mader}
Let $G$ be a graph with $n$ vertices and $e(G)$ edges and without a 
$(k+1)$-connected 
subgraph. 
\begin{itemize}
 \item If $n\geq k^2+k$, then $e(G)\leq \frac{3k-1}{2} 
(n-k)$;
\item If $2k \leq n \leq k^2+k$, then $
                                      e(G)\leq\frac{3k-1}{2} 
(n-k) +\frac{n-k}{3k}+\frac{k^2}{3}\cdot 
\frac{k}{2(n-k)} -\frac{k}{2} 
                                      $.
\end{itemize}
\end{con}
 
As shown in \autoref{mader_extremal}, the graphs $H_i$ attain the bounds in 
\autoref{precise_mader}.

\bibliographystyle{plain}
\bibliography{literatur}

\end{document}

%% file: minimalJ.tex
\usepackage{epsfig}
\usepackage{graphicx}
\usepackage{amsbsy}
\usepackage{amsmath}
\usepackage{amsfonts}
\usepackage{amssymb}
\usepackage{textcomp}
\usepackage{hyperref}
\usepackage{aliascnt}

\newcommand{\mcm}[3]{\newcommand{#1}[#2]{{\ensuremath{#3}}}} 

\mcm{\tuple}{1}{\langle #1 \rangle}
\mcm{\name}{1}{\ulcorner #1 \urcorner}
\mcm{\Nbb}{0}{\mathbb{N}}
\mcm{\Zbb}{0}{\mathbb{Z}}
\mcm{\Rbb}{0}{\mathbb{R}}
\mcm{\Cbb}{0}{\mathbb{C}}
\mcm{\Qbb}{0}{\mathbb{Q}}
\mcm{\Acal}{0}{\cal A}
\mcm{\Bcal}{0}{\cal B}
\mcm{\Ccal}{0}{\cal C}
\mcm{\Dcal}{0}{\cal D}
\mcm{\Ecal}{0}{\cal E}
\mcm{\Fcal}{0}{\cal F}
\mcm{\Gcal}{0}{\cal G}
\mcm{\Hcal}{0}{\cal H}
\mcm{\Ical}{0}{\cal I}
\mcm{\Jcal}{0}{\cal J}
\mcm{\Kcal}{0}{\cal K}
\mcm{\Lcal}{0}{\cal L}
\mcm{\Mcal}{0}{\cal M}
\mcm{\Ncal}{0}{\cal N}
\mcm{\Ocal}{0}{{\cal O}}
\mcm{\Pcal}{0}{{\cal P}}
\mcm{\Qcal}{0}{{\cal Q}}
\mcm{\Rcal}{0}{{\cal R}}
\mcm{\Scal}{0}{{\cal S}}
\mcm{\Tcal}{0}{{\cal T}}
\mcm{\Ucal}{0}{{\cal U}}
\mcm{\Vcal}{0}{{\cal V}}
\mcm{\Wcal}{0}{{\cal W}}
\mcm{\Xcal}{0}{{\cal X}}
\mcm{\Ycal}{0}{{\cal Y}}
\mcm{\Mfrak}{0}{\mathfrak M}

\mcm{\restric}{0}{\upharpoonright}
\mcm{\upset}{0}{\uparrow}
\mcm{\onto}{0}{\twoheadrightarrow}
\mcm{\smallNbb}{0}{{\small \mathbb{N}}}
\DeclareMathOperator{\preop}{op}
\mcm{\op}{0}{^{\preop}}

{\begin{array}{c}
\setlength{\unitlength}{1em}}%
{\end{array}}

\usepackage{amsthm}

\newcommand{\theoremize}[2]{\newaliascnt{#1}{thm} \newtheorem{#1}[#1]{#2} \aliascntresetthe{#1}}

\theoremstyle{plain}
\newtheorem{thm}{Theorem}[section]
\theoremize{lem}{Lemma}
\theoremize{skolem}{Skolem}
\theoremize{fact}{Fact}
\theoremize{sublem}{Sublemma}
\theoremize{claim}{Claim}
\theoremize{obs}{Observation}
\theoremize{prop}{Proposition}
\theoremize{cor}{Corollary}
\theoremize{que}{Question}
\theoremize{oque}{Open Question}
\theoremize{con}{Conjecture}

\theoremstyle{definition}
\theoremize{dfn}{Definition}
\theoremize{rem}{Remark}
\theoremize{eg}{Example}
\theoremize{exercise}{Exercise}
\theoremstyle{plain}